\documentclass{amsart}
\usepackage{graphics,verbatim,color,amsthm}

\usepackage[all]{xy}

\theoremstyle{plain}
\newtheorem{theorem}{Theorem}

\newtheorem{lemma}{Lemma}  
 
\newtheorem{definition}{Definition}

\theoremstyle{remark}
\newtheorem*{remark}{Remark}

\theoremstyle{definition}
\newtheorem{example}{Example}

\def\cal{\mathcal}

\def\metric#1{\langle #1 \rangle}
\def\pairing#1{\langle #1 \rangle}

\def\norm#1{|#1|}

\def\metrictwo#1{\langle\!\langle #1 \rangle\!\rangle}

\def\normtwo#1{\| #1 \|}

\def\R{\mathbb{R}}

\def\RP{\mathbf{RP}}
\def\C{\mathbb{C}}

\def\CP{\mathbf{CP}}
\def\S{\mathbf{S}}
\def\l{\lambda}
\def\a{\alpha}
\def\b{\beta}

\def\g{\gamma}

\def\k{\kappa}
\def\s{\sigma}

\def\t{\theta}
\def\r{\rho}

\def\z{\zeta}

\def\bP{\mathbb{P}}

\def\cV{{\mathcal V}}
\def\cW{{\mathcal W}}
\def\cC{{\mathcal C}}
\def\cS{{\mathcal S}}

\def\tensor{\otimes}
\def\into{\rightarrow}

\def\intersection{\cap}

\def\dim{{\rm dim}\,}

\def\conj#1{\overline{#1}}

\DeclareMathOperator{\diag}{diag}

\DeclareMathOperator{\re}{re}
\DeclareMathOperator{\im}{im}
\DeclareMathOperator{\id}{id}
\DeclareMathOperator{\deter}{det}

\def\smallskip{\par\vspace{1mm}}
\def\medskip{\par\vspace{2mm}}
\def\bigskip{\par\vspace{3mm}}

\def\fr#1#2{\frac{#1}{#2}}
\def\smfr#1#2{\tfrac{#1}{#2}}

\def\m#1{\begin{bmatrix}#1\end{bmatrix}}

\newcount\equationcount
\def\thenumber{0}
\def\eq#1{\global\advance\equationcount by 1
   \def\thenumber{\number\equationcount}
                        {$$#1\eqno(\thenumber)$$}}

\def\interl{\mathop{\vrule width6pt height0.4pt depth0pt
\vrule depth0pt height6pt}}

\def\bP{\mathbb{P}}

\def\CP{\mathbb{CP}}

\newcommand{\dd}[2]
{
{{\partial #1}   \over {\partial #2}}
}

\begin{document}

\title[Planar Three-Body Problem]{Symmetric Regularization, Reduction and Blow-Up of the Planar Three-Body Problem}
\author{Richard Moeckel}
\author{Richard Montgomery}
\address{School of Mathematics\\ University of Minnesota\\ Minneapolis MN 55455}

\email{rick@math.umn.edu}

\address{Dept. of Mathematics\\ University of California, Santa Cruz\\ Santa Cruz CA}

\email{rmont@count.ucsc.edu}

\date{February 3, 2012 (Preliminary Version)}



\thanks{}

\begin{abstract}
We carry out a sequence of coordinate changes for the planar three-body problem which successively eliminate the translation and rotation symmetries, regularize all three double collision singularities and blow-up the triple collision.  Parametrizing the configurations by the three relative position vectors maintains the symmetry among the masses and simplifies the regularization of binary collisions.   Using size and shape coordinates facilitates the reduction by rotations and the blow-up of triple collision while emphasizing the role of the shape sphere.   
By using homogeneous coordinates to describe Hamiltonian systems whose configurations spaces are spheres or projective spaces, we are able to take a modern, global approach to these familiar problems.  We also show how to obtain the reduced and regularized differential equations in several convenient local coordinates systems. 
\end{abstract}

\maketitle

\section{Introduction and History} 

The three-body problem of Newton 
has  symmetries and singularities. 
The  reduction process  eliminates  symmetries  
thereby reducing the number of degrees of freedom.
The Levi-Civita regularization  eliminates binary collision singularities
by a non-invertible coordinate change together with a time reparametrization.
The McGehee blow-up  eliminates  the triple collision singularity
by an  ingenious polar coordinate change and another time reparametrization.
Each  process has been applied individually and in various combinations
to the three-body problem, many times.

In this  paper we apply all three processes globally and systematically,
with no one body singled out in the various transformations.  
The end result is a complete flow on a five-dimensional manifold with boundary.   
We focus  attention on the geometry of the various spaces and  maps appearing along the way.
At the heart of this paper is a beautiful degree-4 octahedral covering map of the shape sphere, branched over 
the binary collision points (see figure~\ref{fig_lemaitre}).   This map first appears in the work of Lemaitre \cite{Lemaitre, Lemaitre2}.
One of our goals is to give a modern, geometrical approach to this regularizing map.

The reduction procedure for the three body problem dates back  to Lagrange  \cite{Lagrange}    
who found elegant differential equations for 10 translation and rotation invariant variables,  including the 
 the squares of the lengths of the three sides of the triangle formed by the bodies.   These equations are valid for the three-body problem in any
 dimension.   The variables of Lagrange also have the advantage of maintaining the symmetry among the masses.  
 On the other hand,  for the planar problem they are subject to 3 nonlinear constraints in addition to the 
 energy and angular momentum integrals.  Moreover, we do not know a way to regularize the
 binary collision singularities in Lagrange's equations.  For a modern introduction to Lagrange's equations see
 \cite{Albouymutual, Albouy-Chenciner, Chenciner}.

Jacobi  eliminates the translation symmetry by the familiar device of fixing the center of mass at the origin and introducing
Jacobi coordinates \cite{Jacobi}.  The elimination of rotations is achieved by introducing some angular variable (or variables in the spatial case) to describe the overall rotation of the triangle together with some complementary, rotation-invariant variables.  This method, which is the basis for much of the later work on the three-body problem,  has some disadvantages.  First, the Jacobi coordinates break the symmetry among the masses, making it much more difficult to regularize all three binary collisions at once.  Second, for topological reasons, there is no way to choose an angular variable suitable for a global reduction which includes the binary collision configurations,  Namely, the map from the normalized configuration space to the shape sphere is a Hopf fibration, a nontrivial circle bundle.  If we delete the binary collision points, the bundle becomes trivial but this deletion is not conducive to subsequent regularization.

Murnaghan \cite{Murnaghan}  derived a symmetrical Hamiltonian for the planar three-body problem in terms of the lengths of the sides and an angular variable representing the overall rotation of the triangle with respect to an inertial coordinate system.  Then he obtains a reduced system by ignoring the angular variable.  Van Kampen and  Wintner carry out a similar reduction for the spatial three-body problem \cite{VanKampenWintner}.  While these reductions avoid breaking the symmetry, they are still subject to the problem about the use of angular variables in a nontrivial bundle. In addition, using the side lengths as variables leads to  differential equations which are not smooth at the collinear configurations (a problem seemingly avoided somehow by Lagrange).

Lemaitre \cite{Lemaitre} introduced a symmetrical approach to reduction and regularization of binary collisions leading to the octahedral branched covering map of the sphere mentioned above.  After using Euler angles to reduce by rotations, he introduces a size variable and two shape variables which can be viewed as spherical coordinates on the shape sphere which we use below.  The regularization of binary collisions is done in the shape variables by means of the octahedral covering map.  The use of Euler angles limits the validity of the reduction step of Lemaitre's work and the derivations are based on rather heavy trigonometric computations.    But much of this paper can be viewed as a modern, global way to arrive at his covering map.

In this endeavor we  have the advantage of the modern theory of reduction of Hamiltonian systems with symmetry.  Smale describes the reduction process for the three-body problem as the formation of a quotient manifold with a reduced Hamiltonian flow \cite{Smale}.
Meyer \cite{Meyer} and   Marsden-Weinstein \cite{Marsden-Weinstein} formalized the reduction 
procedure into what  is now called   ``symplectic reduction theory".  Fixing the integrals of motion determines invariant manifolds in phase space.  The quotient spaces of these invariant manifolds are the reduced phase spaces and the flows induced on them are again Hamiltonian with respect to an appropriate symplectic structure and a reduced Hamiltonian function.  

The regularization procedure goes back to Levi-Civita \cite{LeviCivita}
who showed how to regularize binary collisions in perturbed planar Kepler problems by using the complex squaring map (a branched double covering of the complex plane).    It is easy to adapt his method to regularize one of the binary collisions in the three-body problem, but regularizing all three requires more ingenuity.   Lemaitre's regularizing map behaves like the complex squaring map at each of the binary collision points on the shape sphere.  Another approach to simultaneous regularization (without reduction) was introduced by Waldvogel \cite{Waldvogelreg} who used a quadratic mapping of the translation-reduced configuration space $\C^2$.  We use a similar quadratic mapping
applied to certain homogeneous shape variables below.   Heggie \cite{Heggie} found an elegant,  symmetrical way to
regularize all of the binary collisions for the $N$-body problem.  In the planar case, his method is to apply separate Levi-Civita transformations to each of  the difference vectors $q_i -q_j$.   We apply this same idea below, but to the homogeneous shape variables, where it is found to induce Lemaitre's octahedral covering.

Triple collision acts like an essential singularity in the three-body problem.
McGehee in 1974 \cite{McGehee}  showed how  an extension of spherical coordinates,   together
with a time reparametrization yields a flow with no singularities at triple collision. 
This ``McGehee blow-up'' has  the effect of replacing the triple collision point  by  a manifold called the collision manifold.
Relative to the new parameterization, it  takes forever to reach triple collision, whereas
the  Newtonian time to triple collision is finite.  The  flow on the triple collision manifold 
 governs  the behavior of  near-triple collision solutions.    
One aspect of the blow-up procedure is the use of separate size and shape coordinates to describe the configuration of the bodies.  As shown below, such a splitting also facilitates the global reduction by rotations.    

Several authors have combined blow-up of triple collision with reduction and/or regularization of binary collision.  Waldvogel reduced and regularized the flow on the zero-angular-momentum triple collision manifold \cite{Waldvogeltriple}.  The first part of his paper combines Murnaghan's reduction procedure with some formulas of Lemaitre to obtain a reduced and regularized Hamiltonian for the zero-angular momentum three-body problem.  Binary collisions are not regularized on the nonzero angular momentum levels.  However,  it is known that triple collisions can only occur when the angular momentum is zero.  After restricting to the zero angular momentum manifold,  Waldvogel blows up the triple collision to get reduced, regularized and blown-up differential equations.   Simo and Susin used these coordinates in their study of the dynamics on the collision manifold \cite{SimoSusin}.     These coordinates are very much in the spirit of this paper but do not achieve a full reduction, regularization and blow-up due to the restriction to zero angular momentum.

The present paper draws on all these sources.  We begin with some symplectic reduction theory.  Turning to the three-body problem we eliminate translation symmetry by introducing the three difference vectors $Q_{ij}= q_i-q_j$ as coordinates.  Since these are linearly dependent, some effort is needed to justify the change of coordinates.  Next we introduce a size variable and an associated spherical coordinates $X_{ij}$.  One novelty of our approach is that we use homogeneous coordinates to describe points on spheres.  Instead of constraining the spherical coordinates to have a fixed norm, we only ask them to avoid the origin and then we find differential equations for them which are invariant under scaling.  

Once this point of view is adopted, it is relatively easy to carry out a global reduction by rotations.  Using complex coordinates, the combined action of scaling and rotation is just scaling by a complex number.  Quotienting by complex scaling we end up with a complex projective space, in fact with $\CP^1$.  Of course as  real manifolds $\CP^1\simeq\S^2$ and this is our version of the shape sphere.  We finally obtain a global reduction of the planar three-body problem with a six-dimensional reduced phase space, the cotangent bundle of $\R^+\times \S^2$.

Turning to regularization, we use simultaneous Levi-Civita transformations of the homogeneous variables $X_{ij}$ to regularize all three binary collisions.  This regularizing map is applied to both the rotation-reduced and unreduced problems.  In the reduced case we get a reduced and regularized system on the cotangent bundle of $\R^+\times \S^2$ which is related to the unregularized version by Lemaitre's map.

Finally we show how McGehee's blow-up procedure can be applied to the various Hamiltonians we have found.

 \section{Symplectic Reduction}\label{sec_symplecticreduction}
In this section we recall some results about the reduction of a Hamiltonian system with symmetry.  In addition we show how to tell when two symmetric Hamiltonian systems lead to equivalent reduced systems. 

First we describe the basic symplectic reduction theory  of Meyer \cite{Meyer} and Marsden-Weinstein \cite{Marsden-Weinstein} in the case of a system with symmetry.  Suppose $(M,\omega)$ is a symplectic manifold and $G$ is a Lie group which acts on $M$ as a group of symplectic diffeomorphisms.  Let $J:M\into \mathfrak{g}^*$ be the momentum map, where $\mathfrak{g}^*$ is the dual of the Lie algebra of $G$.  If we fix a momentum value $\mu\in\mathfrak{g}^*$ and suppose that the action of $G$ maps the level set $J^{-1}(\mu)$ into itself.
The quotient space
$$P_\mu = J^{-1}(\mu)/G$$ is called the {\em reduced phase space}.

If the group action is free and proper, then this space is a smooth manifold.  There is an induced symplectic form 
$\omega_\mu$ on $P_\mu$ which is obtained as follows.  First, for $x\in M$, restrict $\omega(x)$ to the tangent spaces $T_x J^{-1}(\mu)$.  The resulting two-form has a kernel which is precisely the tangent space to the group orbit through $x$.  This implies that there is an induced two-form on the quotient vector space which is the tangent space to the quotient manifold. 

Now if $H:M\into \R$ is a $G$-invariant Hamiltonian then the corresponding Hamiltonian flow has $J^{-1}(\mu)$ as an invariant set and $G$-orbits map to $G$-orbits under the flow.  Hence there is a well-defined quotient flow on $J^{-1}(\mu)/G$.   There is also a reduced Hamiltonian $H_\mu:P_\mu\into \R$ and the reduction theorem states that the quotient flow on $(P_\mu,\omega_\mu)$ is the Hamiltonian flow of the reduced Hamiltonian.

Now suppose we have two such Hamiltonian systems with symmetry.  For $i = 1,2$, there will be symplectic manifolds $(M_i,\omega_i)$, symmetry groups $G_i$ and momentum maps $J_i$.  If we fix momentum values $\mu_i$ we get reduced reduced phase spaces $P_i = J_i^{-1}(\mu_i)/G_i$ with symplectic forms $\omega_{\mu_i}$.   Suppose $H_i:M_i\into\R$ are $G_i$-invariant Hamiltonians and let $H_{\mu_i}:P_i\into \R$ be the corresponding reduced Hamiltonians.  We want to give a concrete way to check that the two reduced Hamiltonian flows are equivalent.

Suppose we have a smooth map $F: J_1^{-1}(\mu_1)\into J_2^{-1}(\mu_2)$ which maps $G_1$-orbits into $G_2$-orbits, i.e., $F$ is equivariant.  Then $F$ induces a smooth map of quotient manifolds $\hat F:P_1\into P_2$.  
We will call $F$  {\em partially symplectic} if it preserves the restrictions of the symplectic forms, i.e., $F^*(\omega_2|_{J_2^{-1}(\mu_2)})= \omega_1|_{J_1^{-1}(\mu_1)}$.  It follows that $\hat F:(P_1,\omega_{\mu_1})\into (P_2,\omega_{\mu_2})$ is symplectic.  Hence $\hat F$ is a local diffeomorphism, even if $F$ itself is locally neither injective nor surjective.  Then the usual theory of symplectic maps applied to $\hat F$ gives:

\begin{theorem}\label{th_tworeductions}
Suppose $F: J_1^{-1}(\mu_1)\into J_2^{-1}(\mu_2)$ is a partially symplectic, equivariant map and that the restrictions of the Hamiltonians are related by $H_1= H_2\circ F$.  Then $\hat F:P_1\into P_2$ is a symplectic, local diffeomorphism of the reduced phase spaces which takes orbits of the reduced Hamiltonian flow of $H_{\mu_1}$ to those of $H_{\mu_2}$.  
\end{theorem}

\begin{definition}\label{def_partialinverse}
A partially symplectic, equivariant map $G: J_2^{-1}(\mu_2)\into J_1^{-1}(\mu_1)$ such that $F\circ G = \id (\bmod \,G_2)$ and $G\circ F = \id (\bmod\,G_1)$ (so that these maps take group orbits into group orbits)  will be called a {\em pseudo-inverse} for $F$. 
\end{definition} 
A partial inverse $G$ for $F$ induces a bona fide inverse $\hat G$   for $\hat F$ 
which exhibits an equivalence between the two  reduced Hamiltonian flows.

As a special case, suppose the two Hamiltonians are both defined on the same space and have the same symmetry group. If their restrictions to $J^{-1}(\mu)$ agree then they will lead to the same reduced system.  The identity map will provide the required partially symplectic map.  We will call two such Hamiltonians {\em equivalent}.  Equivalent Hamiltonians may produce different flows on $J^{-1}(\mu)$ but the quotient flows will agree.

The following theorems  about the symplectic reduction of a cotangent bundle $M = T^*X$ will be used later.
(See   \cite{Abraham-Marsden} Theorem 4.3.3 for a version of these theorems.) Suppose   $G$ acts freely on the configuration space  $X$   and that the $G$-action on $M$  is the canonical lift of this base action.  Suppose that the orbit space $B$ for  the $G$ action on $X$
is  a manifold  and the  projection $\pi : X \to B$ a submersion.

\begin{theorem}\label{th_cotangentreduction0}
Under the above assumptions, the reduced space  $P_0$ of
$T^*X$ at $\mu = 0$ is
isomorphic to $T^*B$ with its canonical symplectic structure $\omega_B$
\end{theorem}
The theorem can be proved as  a special case  of Theorem \ref{th_tworeductions}.
Because $\pi$ is onto, $d \pi_x: T_x X \to T_{\pi(x)} B$ is an onto linear map for each $x \in X$. 
Consequently the dual map $d \pi_x ^{*}: T^* _{\pi(x)} B \to T^*_x X$ is injective.  
In the next paragraph we will show that the image of this dual is  $J^{-1}(0)_x$:
\begin{equation}
im(d \pi_x ^*) = J^{-1} (0)_x := J^{-1}(0) \cap T^*_x X.
\label{J0}
\end{equation}
It follows that  
we can invert $d \pi_x ^*$ on the fiber $J^{-1}(0)_x \subset  T^* _x X$.
Define  
$$F: J^{-1} (0) \to T^* B  \; ; F(x,p) = (\pi(x),  d \pi_x ^{*-1} (p)) .$$
One verifies that $F$ is a partially symplectic map  relative to    $G$ acting  on $J^{-1} (0)$, and
the trivial group acting on $T^*B$.   
A particularly  easy way to see the partially symplectic nature of $F$
is to introduce  local bundle coordinates  $X \supset \pi^{-1} (U)  \cong U \times G $.
($X$ is covered by sets of this nature.  ) In bundle coordinates
$\pi (x,g) = x$  and so  $T_U ^*X \cong T^*U \times G \times \mathfrak{g}^*$.
We write elements of  $T^* X$ over $U$ as   $(b, P; g, \mu), b \in U, P \in T^* _b U,  g \in G,  \mu \in \mathfrak{g}^*$.
In these coordinates $J(b, P; g, \mu) = \mu$,  so that the
general element of $J^{-1}(0)_U $ can be written    $(b, P_b , g, 0)$  and $F(b, P_b , g, 0) = (b, P_b)$. 
We have  $\omega_X = dx \wedge dP + dg \wedge d \mu$ and, 
$\omega_B = dx \wedge dP$, where we hope the meaning of these symbolic expressions is obvious.
It follows immediately that $F^* \omega_B =  \omega_X|_{J^{-1}(0)}$  which
is the claimed    partially symplectic nature
of $F$.    Theorem~\ref{th_cotangentreduction0} follows.

We explain why eq (\ref{J0}) holds, and in the  
process gain some understanding of the  momentum map.
The group action is a map $ G \times X \to X$ which, when differentiated with respect to $g \in G$ at
the identity  
yields the ``infinitesimal action''   $\sigma: \mathfrak{g} \times X  \to TX$.
For each frozen $x$, the map $\sigma_x  : \mathfrak{g} \to T_x X$ is linear and, because $G$
acts freely,  injective.  As we   vary $x$, $\sigma$ forms  a 
 vector bundle map,  part of an exact sequence  of vector bundle maps
over $X$: 
$$
\xymatrix{
0 \ar[r] & \mathfrak{g} \times X  \ar[r]^{\sigma} & TX \ar[r]^{d\pi} & \pi^* TB \\
}
$$
where $\pi^* TB = \{ (x, V);  x \in X,  V \in T_{\pi(x)} B \}$ is the pull-back of $TB$ over $B$ by the map $\pi: X \to B$.
(Exactness of the sequence follows by differentiating the statement that the fibers of $\pi$ are the $G$-orbits.)
Dualizing we get  
 $$\xymatrix{
0 & \ar[l]  \mathfrak{g^*} \times X  & \ar[l]^{\sigma^*} T^*X  &  \ar[l]^{d\pi ^*}  \pi^*T^*B  \\
}$$  
The momentum map for the $G$-action on $T^*X$ is $\pi_1 \circ \sigma^*$ where
$\pi_1: \mathfrak{g^*} \times X \to \mathfrak{g^*} $ is the projection onto the first factor.
In other words
$$J(x,p) = \sigma_x ^* p.$$
It follows from the exactness of the dual sequence that  $im(d \pi_x ^*) = ker(\sigma_x ^*)$
which is precisely equation (\ref{J0}).

In order to identify the   reduction of $M = T^*X$  at  a non-zero value, $\mu \ne 0$,
we introduce a connection $\Gamma$ for the bundle $G \to X \to B$.  The  curvature of the connection $\Gamma$
is a $\mathfrak{g}$-valued two-form $\Omega$ on $B$, which we may pull-back to $T^*B$ via the canonical projection $\tau_B: T^*B \to B$.  Then $\mu \cdot \Omega$ is a scalar-valued two-form on $B$.
 \begin{theorem}\label{th_cotangentreductionmu}
Under the same  assumptions as above on $G$,  the reduced space $P_\mu$ of  $T^*X$  at  $\mu$ is
isomorphic to $T^*B$ with the twisted symplectic structure $\omega_B - \tau_B^* \mu \cdot \Omega$
\end{theorem}

We only present the proof in the case $G = S^1$, whose Lie algebra we identify with $\R$
in the usual way.  Then a connection is a $G$-invariant  one-form on $T^*X$ which
satisfies the normalization property:
$$J(x, \Gamma (x)) = 1.$$
Its curvature $\Omega$ is defined by
$$ d \Gamma = \pi^* \Omega.$$
We define the momentum shift map: 
$$\Phi_{\mu} : J^{-1} (0) \to J^{-1} (\mu) \; ;  \Phi_{\mu} (x, p) = (x , p + \mu \Gamma (x)).$$
which adds $\mu \Gamma $ pointwise to each covector. 
The fiber-linearity of $J$ shows that $\Phi_{\mu}$ does indeed map $J^{-1}(0)$ onto
$J^{-1} (\mu)$. (The  inverse of $\Phi_{\mu}$  subtracts $\mu \Gamma$. ) The map is $G$-equivariant
since $\Gamma$ is $G$-invariant. Thus $\Phi_{\mu}$ induces a $G$-equivariant diffeomorphim
$J^{-1} (0)/G \to J^{-1} (\mu)/G$.  
We have already identified $J^{-1} (0)/G$ with $T^* B$. However, $\Phi_{\mu}$  is not partially
symplectic, so we cannot directly  apply theorem 1.  To understand and quantify this failure, let   
$\Theta = PdQ$ denote the canonical one-form on $T^*X$. Compute  
$\Phi_{\mu}^* \Theta = \Theta + \mu \tau_X ^* \Gamma$.  Taking the exterior derivative,
using $\omega_X = - d \Theta$ we find that $\Phi_{\mu} ^* \omega_X = \omega_X - \mu \tau_X ^* \pi^* \Omega$.
This equation implies  that if we shift the canonical two-form on
$J^{-1}(0)$ by subtracting $\mu \tau_X ^* \pi^* \Omega$
then $\Phi_{\mu}$ is a partially symplectic map between $J^{-1} (0)$ and $J^{-1} (\mu)$.
Theorem \ref{th_cotangentreductionmu} follows.

\section{Reduction by Translations}\label{sec_reductionbytranslations}
To formulate the Newtonian planar three-body problem, it is convenient to use the complex plane where we identify $(x,y)\in\R^2$ with $x+i\, y\in\C$.  

 Let  $q_i\in \C, i=1,2,3$ be the positions of the three-bodies and let 
$q=(q_1,q_2,q_3)\in\C^3$.  We will adopt the Hamiltonian point of view where the conjugate momentum variables $p_i$ are covectors rather than vectors.  If we identify a covector $(a,b)\in\R^{2*}$ with $a+i\, b\in\C$ then we have momentum variables 
$p_i\in\C^*\simeq\C$ and $p=(p_1,p_2,p_3)\in\C^{3*}$.

The planar three-body problem is the Hamiltonian system on the phase space $(\C^3\setminus \Delta)\times \C^{3*}$ with Hamiltonian
\begin{equation}\label{eq_Hamiltonianqp}
\begin{aligned}
H(q,p) &= K_0(p) - U(q)\\
K_0(p) &=  \fr{|p_1|^2}{2m_1}+ \fr{|p_2|^2}{2m_2}+ \fr{|p_3|^2}{2m_3}\\
U(q) &=  \fr{m_1 m_2}{|q_1-q_2|}+ \fr{m_3 m_1}{|q_3-q_1|}+ \fr{m_2 m_3}{|q_2-q_3|}
\end{aligned}
\end{equation}
where $\Delta = \{q: q_i = q_j \; \text{for some}\; i\ne j\}$, the singular set.  From now on, we will not explicitly mention that the singular set must be deleted from the domains of the various Hamiltonians we construct.

The Newtonian potential is invariant under the group $G=\C$ acting by translation on the position vectors and leaving the momenta fixed.  The momentum map is given by
$$p_{tot} = p_1+p_2+ p_3 \in\C^{*}.$$
By fixing a value of this integral and passing to the quotient space one obtains a reduced Hamiltonian system.
A simple and familiar way to accomplish this reduction is to assume $p_{tot}= 0$ and then fix the center of mass at the origin: $m_1q_1+m_2q_2+m_3 q_3 = 0$.

However, we will now describe an alternative method for eliminating the translation symmetry which will make it easier to regularize double collisions later on.   This approach is a variation on the one used by Heggie in \cite{Heggie}.  We will view it as an application of theorem~\ref{th_tworeductions}.

\subsection{Relative coordinates}\label{sec_relative}
Introduce relative position variables $Q_{12}, Q_{31},Q_{23}\in \C$ and corresponding momentum variables  $P_{12}, P_{31},P_{23}\in\C^*$.  The relative coordinates are related to the positions variables $q_i$ by a linear map $Q = Lq$
\begin{equation}\label{eq_Lmap}
L: \C^3 \into \C^3  \qquad \quad  Q_{12} = q_1-q_2\quad Q_{31} = q_3-q_1\quad Q_{23} = q_2-q_3.
\end{equation}
The dual map, which describes the pull-back of the relative momenta $P_{ij}$ to $p$ space,  is given by 
\begin{equation}\label{eq_Lstarmap}
L^*: \C^{3*} \into \C^{3*}\qquad \quad  p_1 = P_{12}-P_{31}\quad p_2 = P_{23}-P_{12} \quad p_3 = P_{31} -P_{23}.
\end{equation}
We naturally have  $Q_{ji} = - Q_{ij}$ and consequently $P_{ji} = -P_{ij}$ so that  eq (\ref{eq_Lstarmap})
can   be written $p_i = \Sigma_j P_{ij}$, a form which extends to the $N$-body problem. 

 The linear map $L$ is neither one-to-one nor onto.  Its kernel,
$$\ker L = \{q: q = (c,c,c)\;\text{for some}\; c\in\R^2= \C\},$$
is  the subspace of translation symmetries  in $q$-space.  So its image 
$$\cal W =  \im L =  \{Q: Q_{12}+Q_{31}+Q_{23} = 0\}$$
is isomorphic to the quotient space of $\C^{3}$ by translations.   $\cal W$ is a complex subspace of $\C^3$ with complex dimension two, or real dimension 4.
We can define a map in the other direction, $q=L^\dagger(Q)$
 \begin{equation}\label{eq_Ldagger}
L^\dagger:\quad  q_1 = \fr{m_2Q_{12}-m_3Q_{31}}{m},q_2 = \fr{m_3Q_{23}-m_1Q_{12}}{m},q_3 = \fr{m_1Q_{31}-m_2Q_{23}}{m}
\end{equation}
$L^\dagger$ maps $\C^3$ onto
$$\cW' = \im L^\dagger =  \{q: m_1q_1+m_2q_2+m_3q_3 = 0\}$$
the zero-center of mass subspace and it is easy to check that the restrictions $L|_{\cW'}$ and $L^\dagger|_\cW$ are inverses.

 For the dual map, we find that the kernel
 is generated by translations in $P$-momentum space
$$\ker L^* = \{P: P = (c,c,c)\;\text{for some}\; c\in\C ^* \}$$
 while the image is the zero-momentum subspace
$$\cal V = \im L^* = \{p: p_1+p_2 + p_n=0\}.$$
The map $L^{\dagger *}:\C^{3*} \into \C^{3*}$
\begin{equation}\label{eq_Lstardagger}
L^{\dagger *}: \qquad  P_{12}= \fr{m_2p_1-m_1p_2}{m},P_{31}= \fr{m_1p_3-m_3p_1}{m},P_{23}= \fr{m_3p_2-m_2p_3}{m}
\end{equation}
maps $\C^{3*}$ onto
$$\cV' = \im L^{\dagger *} =  \{P: m_3P_{12} +m_2P_{31} + m_1 P_{23} =  0\}$$
and the restrictions $L^*|_{\cV'}$ and $L^{\dagger *}|_\cV$ are inverses. 

Define a relative coordinate Hamiltonian   on the $(Q,P)$ phase space $\C^{3}\times \C^{3*}$ by
\begin{equation}\label{eq_HamiltonianPQ3bp}
\begin{aligned}
H_{rel}(Q,P) &= K(P) - U(Q)\\
K(P) = K_0(L^*P) &= \fr{|P_{12}-P_{31}|^2}{2m_1} + \fr{|P_{23}-P_{12}|^2}{2m_2} + \fr{|P_{31}-P_{23}|^2}{2m_3} \\
U(Q) &= \fr{m_1m_2}{|Q_{12}|} +  \fr{m_3m_1}{|Q_{31}|} + \fr{m_2m_3}{|Q_{23}|},
\end{aligned}
\end{equation}
so that
\begin{equation}
\label{eq_relpositionHam}
H(q, L^* P) = H_{rel} (Lq, P).
\end{equation}
The kinetic energy can be written
\begin{equation}\label{eq_Bmatrix}
K(P) = \fr12 P^T B P\qquad B = \m{ (\fr{1}{m_1} +\fr{1}{m_2})I& -\fr{1}{m_1}I &-\fr{1}{m_2}I \\
-\fr{1}{m_1}I &(\fr{1}{m_3}+\fr{1}{m_1})I &-\fr{1}{m_3} I \\
-\fr{1}{m_2}I&-\fr{1}{m_3}I &(\fr{1}{m_2} +\fr{1}{m_3})I    }
\end{equation}
where $I$ denotes the $2\times 2$ identity matrix.

\subsection{Equivalance to the translation-reduced three-body problem}
We will now show that the reduction of the Hamiltonian system with Hamiltonian $H_{rel}(Q,P)$ by translations in momentum space is equivalent to the reduction of the three-body Hamiltonian $H$ by translations in configuration space.

\begin{theorem}\label{th_translationreduction}
 $\cW \times \C^{3*}$ is invariant under the Hamiltonian flow of $H_{rel}(Q,P)$.  The restricted flow  is invariant  under translations in momentum space and it induces a quotient flow which is conjugate to the zero total momentum flow of the three-body problem reduced by translations.\label{th_translation}
\end{theorem}

The proof will be an application of theorem~\ref{th_tworeductions}.  First we describe how the relevant symplectic structures look in complex coordinates.  If $Q\in\C^3$ and $P\in\C^{3*}$ it is convenient to define a Hermitian variant of the natural evaluation pairing:
\begin{equation}\label{eq_pairing}
\metric{P,Q} = \bar P_{12}Q_{12}+\bar P_{31}Q_{31}+\bar P_{23}Q_{23}.
\end{equation}
As a result, if $Q_{jk} = x_{jk}+i\,y_{jk}$ and $P_{jk} = a_{jk}+i\,b_{jk}$ we get
\begin{equation}\label{eq_realimag}
\begin{aligned}
\re \metric{P,Q}&= a_{12}x_{12}+b_{12}y_{12}+\ldots\\
\im \metric{P,Q} &= a_{12}y_{12}-b_{12}x_{12}+\ldots 
\end{aligned}
\end{equation}
Thus the real part of the complex pairing agrees with the usual real pairing and, as a bonus,  the imaginary part is $-\mu$ where $\mu$ is the angular momentum.  With this convention, the canonical one-forms on $(q,p)$-space and $(Q,P)$-space can be written
\begin{equation}\label{eq_canonicaloneforms}
\begin{aligned}
\theta &= \re\pairing{p,dq} = \re(\bar p_1\,dq_1 +  \bar p_2\,dq_2 +  \bar p_3\,dq_3) \\
 \Theta &= \re\pairing{P,dQ}=\re(\bar P_{12}\,dQ_{12}+\bar P_{31}\,dQ_{31}+ \bar P_{23}\,dQ_{23}).
\end{aligned}
\end{equation}

\begin{proof}[Proof of theorem~\ref{th_translationreduction}]
For the three-body problem we have the phase space $M_1 = \C^6\times\C^{6*} = \{(q,p)\}$  with the standard symplectic structure.  The Hamiltonian $H(q,p)$ is invariant under the action of the group $G_1=\C$ acting by 
$c\cdot(q,p) = (q_1+c, q_2+c, q_3+c, p_1,p_2,p_3), c\in\C$.  We fix the momentum level $p_{tot}=0$ and obtain a quotient Hamiltonian flow.

For the Hamiltonian $H_{rel}$ the phase space is  $M_2 = \C^3\times\C^{3*} = \{(Q,P)\}$  with the standard symplectic structure.  $H_{rel}(Q,P)$ is invariant under the action of the group $G_2=\C^{*}$ acting on by 
$c \cdot(Q,P) = (Q_{12}, Q_{31}, Q_{23}, P_{12}+c, P_{31}+c, P_{23}+c), c\in\C^{*}$.  
The momentum map is $Q_{tot} = Q_{12}+Q_{31}+Q_{23}$ and we
fix the momentum level $Q_{tot}=0$ giving a second quotient Hamiltonian flow.

To see that these two quotient flows are equivalent we apply Theorem ~\ref{th_tworeductions}.
Define 
$$F(q,p) = (Lq,L^{\dagger *}p) \qquad G(Q,P) = (L^\dagger Q, L^*P).$$
Then, 
$F:\{p_{tot}=0\}  \into \{Q_{tot}=0\}$ and $G:\{Q_{tot}=0\} \into \{p_{tot}=0\}$.
Moreover  $G\circ F(q,p) = c \cdot (q,p)$ where $-c = \fr1m(m_1q_1+m_2q_2+m_3q_3)\in\C$ is the center of mass.  Similarly $F\circ G(q,p) = c\cdot (Q,P)$ where $-c = \fr1m(m_3P_{12}+m_2P_{31}+m_1P_{23})\in\C^*$.  
In other words $G\circ F = \id (\bmod \,G_1)$ and $F\circ G= \id (\bmod \,G_2)$. 

It remains to verify that  
$F$ and $G$ are  partially symplectic.    Consider  the canonical one-forms (\ref{eq_canonicaloneforms}).
From (\ref{eq_Lmap}) and (\ref{eq_Lstardagger}).  We find, for example $F^*\bar P_{12} =  (m_2\bar p_1-m_1\bar p_2)/m$ and
$F^*dQ_{12} = dq_1-dq_2$.  After a bit of algebra we get
$$F^*\Theta =\theta - \re (\bar p_{tot}(m_1dq_1+m_2dq_2+m_3dq_3)/m).$$
Restricting to $\{p_{tot}=0\}$ shows that $F$ is partially symplectic.  Similarly 
$$G^*\t= \Theta- \re( (m_3\bar P_{12}+m_2\bar P_{31}+m_1\bar P_{23})(dQ_{12}+dQ_{31}+dQ_{23}))/m$$
which we restrict to $\{Q_{tot}=0\}$ to see that $G$ is also partially symplectic.
We have shown that $F$ and $G$ are   pseudo-inverses   in the sense of    definition \ref{def_partialinverse}.
According to eq. (\ref{eq_relpositionHam}) these pseudo-inverses  intertwine $H$ and $H_{rel}$.  
The hypothesis of theorem 1 have been verified, completing the proof.  
\end{proof}

Hamilton's equations for the Hamiltonian $H_{rel}(Q,P)$ are simply
\begin{equation}\label{eq_HamiltonsEqsPQ3bp}
\begin{aligned}
\dot Q  &= B P\\
\dot P &= U_Q = -(\fr{m_1m_2 Q_{12}}{r_{12}^3},\fr{m_3m_1 Q_{31}}{r_{31}^3},\fr{m_2m_3 Q_{23}}{r_{23}^3})
\end{aligned}
\end{equation}
where $r_{ij}= |Q_{ij}|$.
(Note that here and in all of the differential equations below, partial derivatives like $U_Q$ are calculated by simply calculating the corresponding real partial derivatives and converting the resulting real vector or covector to complex notation;  no complex differentiations are involved).
Differential equations for the three-body problem reduced by translations are obtained by restricting $Q$ to $\cW$.  Then $Q$ remains in $\cW$ under the flow.   Moreover, covectors $P,P'$ which are initially equivalent under translation remain so.

Since the symmetry group $\C^*$ acts only on the momenta $P_{ij}$, the reduced phase space is the eight-dimensional space $\cW\times (\C^{3*}/ \C^*)\simeq \cW\times \im L^* = \cW\times \cV$.  This can be identified with the cotangent bundle $T^*\cW = \cW\times \cW^*$ as follows.  Let $P\in\C^{3*}$.  Then $P|_{\cW} \in \cW^*$ and two covectors $P,P'\in\C^{3*}$ have the same restriction to $\cW$  if they differ by an element of $\ker L^*$, i.e., if they are equivalent under the symmetry group.

So far we have not really accomplished any ``reduction'' since there are still $12$ $(Q,P)$ variables.  Essentially, we have traded the constraint $p_{tot}=p_1+p_2+p_3=0$ and the translation symmetry in $q$ for the constraint $Q_{tot} = Q_{12}+Q_{31}+Q_{23} = 0$ and translation symmetry in $P$.  We will see below that the use of the $Q_{ij}$ is advantageous for regularizing double collisions.  A genuine reduction of dimension can be easily achieved by introducing a basis for $\cW$. Moreover, this can be accomplished in several ways as we will see in section~\ref{sec_parametrizingW} below.  But one virtue of (\ref{eq_HamiltonianPQ3bp}) is that it avoids making a choice of parametrization and thereby preserves the symmetry of the problem under permutations of the masses.

\subsection{Mass Metrics and the Kinetic Energy}\label{sec_massmetrics}
The potential energy $U(Q)$ of (\ref{eq_HamiltonianPQ3bp}) is particularly simple, but the kinetic energy $K(P)$ seems less natural.  In this section we will see that it is related by duality to a Hermitian metric which will play an important role later on.

Define a Hermitian  {\em mass metric} on $\C^3$ by
\begin{equation}\label{eq_massmetric}
\metric{V,W}=  \fr 1m \left( m_1 m_2 \bar V_{12}^TW_{12} +m_3 m_1 \bar V_{31}^T W_{31} +m_2 m_3 \bar V_{23}^T W_{23} \right).
\end{equation} 
The corresponding norm is given by
\begin{equation}\label{eq_normQsq}
|Q|^2 = \fr 1m \left( m_1 m_2 |Q_{12}|^2 +m_3 m_1 |Q_{31}|^2+m_2 m_3 |Q_{23}|^2\right).
\end{equation}
The mass norm 
$$r= |Q| = \sqrt{\metric{Q, Q}}$$
provides a natural measure of the size of a configuration $Q= (Q_{12},Q_{31},Q_{23})\in\C^3$.   In particular, $r=0$ represent triple collision.  There is a {\em dual mass metric} on $\C^{3*}$ given by
\begin{equation}\label{eq_dualmassmetric}
\metric{P,R}= m \left( \fr{\bar P_{12}^TR_{12}}{m_1m_2} +\fr{\bar P_{31}^T R_{31}}{m_3 m_1} + \fr{\bar P_{23}^T R_{23}}{m_2 m_3} \right).
\end{equation}
with squared norm
\begin{equation}\label{eq_normPsq}
|P|^2 = m \left( \fr{ |P_{12}|^2}{m_1m_2}+\fr{ |P_{31}|^2}{m_3 m_1}+\fr{ |P_{23}|^2}{m_2 m_3}\right).
\end{equation}
Note:
 Altogether we have  three interpretations of $\metric{.,.}$ depending on whether the arguments are two vectors
 (eq \ref{eq_massmetric}), two covectors  (eq \ref{eq_dualmassmetric}),  or a vector and a covector,
  (eq \ref{eq_pairing}).  All three pairings are Hermitian, begin complex-linear in the second argument and anti-linear in the first.

Introduce the notation $\cW_0 = \cW\setminus 0$ (and a similar notation for any vectorspace).
If $Q\in \cW_0 $ then it is easy to check that the vectors $Q, N, T$ form a Hermitian-orthogonal complex basis for $T_Q\C^3$ with respect to the Hermitian mass metric, where
\begin{equation}\label{eq_QNT}
\begin{aligned}
Q &= (Q_{12},Q_{31},Q_{23})\\
N &= (m_3,m_2,m_1)\\
T &= (\fr{\bar Q_{31}}{m_2}-\fr{\bar Q_{23}}{m_1},\fr{\bar Q_{23}}{m_1}-\fr{\bar Q_{12}}{m_3},\fr{\bar Q_{12}}{m_3}-\fr{\bar Q_{31}}{m_2}).
\end{aligned}
\end{equation}
$Q$ is a radial vector and $N,T$ are respectively normal and tangent to $\cW$.  Clearly $\{Q,T\}$ is a basis for $\cW$.  

The next lemma shows the relationship between the kinetic energy and the dual of the mass metric.
\begin{remark} On terminology.  A nondegenerate quadratic form on a vector space, or on the fibers of 
a vector bundle, determines uniquely a quadratic form on the dual vector space, or on the fibers of the dual vector bundle.
We refer to this dual quadratic form as either the `cometric' or the `dual norm'.
\end{remark}
\begin{lemma}\label{lemma_kinetic}
The kinetic energy satisfies
\begin{equation}\label{eq_kinetic}
K(P) =\fr12\fr{|\pairing{P,Q}|^2}{|Q|^2}+\fr12\fr{|\pairing{P,T}|^2}{|T|^2}= \fr12 |P|^2-\fr12 \fr{|\pairing{P,N}|^2}{|N|^2}= \fr12|\pi_\cW^*P|^2
\end{equation}
where $|P|$ is the dual mass norm and where $\pi_\cW: \C^3\into \C^3$ is orthogonal projection onto $\cW$ with respect to the mass metric.

Moreover, $K(P)$ can be characterized as one-half of the unique translation-invariant quadratic form on $T^*_Q\C^3$ which represent the dual of the restriction of the mass norm to $T_Q\cW$.
\end{lemma}
\begin{proof}
A direct computation shows that
$$|P|^2 - \fr{|\pairing{P,N}|^2}{|N|^2} = \fr{|P_{12}-P_{31}|^2}{2m_1} + \fr{|P_{23}-P_{12}|^2}{2m_2} + \fr{|P_{31}-P_{23}|^2}{2m_3} = 2K(P).$$
On the other hand, dual norms, or `cometrics', can be characterized by the property that for any orthogonal basis $\{Q,N,T\}$, 
$$|P|^2 =\fr{|\pairing{P,Q}|^2}{|Q|^2}+\fr{|\pairing{P,N}|^2}{|N|^2}+\fr{|\pairing{P,T}|^2}{|T|^2}.$$
Hence
$$2K(P) = |P|^2-\fr{|\pairing{P,N}|^2}{|N|^2}=\fr{|\pairing{P,Q}|^2}{|Q|^2}+\fr{|\pairing{P,T}|^2}{|T|^2}.$$
and this is also the formula for $|P\circ \pi_\cW|^2$.

If we view $T^*_Q\cW$ as the quotient space of $T^*_Q\C^3$ under momentum translations, then any norm on $T^*_Q\cW$ is represented by a unique translation-invariant quadratic form on $T^*_Q\C^3$.   In particular, this applies to the dual norm of the restriction of the mass norm to $T_Q\cW$.  
Since $\{Q,T\}$ is an orthogonal basis for $T_Q\cW$ with respect to the mass metric, this ``lift'' of the dual norm will be given by
$$\fr{|\pairing{P,Q}|^2}{|Q|^2}+\fr{|\pairing{P,T}|^2}{|T|^2} = 2K(P).$$
\end{proof}

\subsection{Parametrizing $\cW$}\label{sec_parametrizingW}
Let 
$$e_1= (a_{12},a_{31},a_{23})\qquad e_2= (b_{12},b_{31},b_{23}) \in\cW$$
 be any complex basis for $\cW$.  The corresponding coordinate map  is $f:\C^2\into \cW\subset \C^3$,
$$f(\xi_1,\xi_2) = \xi_1\,e_1 + \xi_2\,e_2\quad\text{or}\quad Q_{ij} = \xi_1\,a_{ij} + \xi_2\, b_{ij}$$
where $\xi = (\xi_1,\xi_2)\in\C^2$ are the new coordinates.  

Extend $f$ to a map
$F:T^*\C^2\into \cW\times\C^{3*}$ by letting $P\in\C^{3*}$ be any solution to the equations
$$\pairing{P,e_1} = \bar\eta_1\qquad \pairing{P,e_2}=\bar\eta_2$$
where $\eta=(\eta_1,\eta_2)\in\C^{2*}$ is the dual momentum to $\xi$ and $N$ is the normal vector to $\cW$ from (\ref{eq_QNT}).   Any two solutions will differ by a momentum translation which will not affect the computations below. This definition makes $F$ partially symplectic, where the symplectic structure on $T^*\C^2$ derives from the canonical one-form 
$$\theta = \re\pairing{\eta,\xi} =\re(\bar \eta_1\,\xi_1 +\bar\eta_2\,\xi_2).$$

To find the new Hamiltonian, note that the pull-back of the Hermitian mass metric  is 
$$\metric{\xi,\xi'} = {\bar\xi}^T G \xi' \qquad\qquad G = \m{g_{11}&g_{12} \cr g_{21}&g_{22}}\quad g_{ij}= \metric{e_i,e_j}.$$
Clearly this can be viewed as the pull-back of the restriction of the mass metric to $\cW$.  
The dual of this metric is
$$\metric{\eta,\eta'} ={ \bar\xi}^T G \xi' \qquad\qquad  G^{-1} = \fr1{g}\m{g_{22}&-g_{21} \cr - g_{12}&g_{11}}\quad g= \deter G.$$
It follows from lemma~\ref{lemma_kinetic} and
the fact that the momenta  also transform as pull-backs that the kinetic energy will be one-half of the dual norm. 

The Hamiltonian becomes
\begin{equation}\label{eq_parametrizedH}
H(\xi,\eta) = \fr12\bar{\eta}^TG^{-1}\eta- U(\xi)
\end{equation}
 where
 $$U(\xi) = \fr{m_1 m_2}{\rho_{12}}+\fr{m_1 m_3}{\rho_{31}}+\fr{m_2 m_3}{\rho_{23}}\qquad \rho_{ij} = |Q_{ij}| = |a_{ij}\xi_1 + b_{ij} \xi_2|.$$

\begin{example}[Heliocentric Coordinates]
One can form such a parametrization of $\cal W$ by choosing one of the masses, 
say $m_1$, to play the role of  the origin.  Set
$$Q_{12}= -\xi_1 \qquad Q_{31}= \xi_2 \qquad Q_{23} = \xi_1- \xi_2$$
so that $\xi_1,\xi_2 \in\C$ are the coordinates of $m_2,m_3$ relative to $m_1$.  
The corresponding basis for $\cW$
$$e_1 = (-1,0,1)\qquad e_2 = (0,1,-1)$$
and the momenta $\bar \eta_i = \pairing{P,e_i}$ satisfy
$$\eta_1 =  P_{23}-P_{12}\qquad \eta_2 = P_{31}-P_{23}.$$
For example, we can choose
$$P_{12}=-\eta_1 \qquad P_{31}= \eta_2 \qquad P_{23}=0.$$
Substituting into $H_{red}$ gives the familiar Hamiltonian
 $$H(\xi,\eta) = \fr{|\eta_1+ \eta_2|^2}{2m_1}+ \fr{|\eta_1|^2}{2m_2}+ \fr{|\eta_2|^2}{2m_3} - 
  \fr{m_1 m_2}{|\xi_1|}-\fr{m_1 m_3}{|\xi_2|}-\fr{m_2 m_3}{|\xi_1-\xi_2|} .$$
\end{example}

\begin{example}[Jacobi Coordinates]
Alternatively one can introduce Jacobi coordinates $\xi_1,\xi_2$ by setting
$$Q_{12}= -\xi_1 \qquad Q_{31}= \xi_2 + \nu_2\xi_1 \qquad Q_{23} = -\xi_2+\nu_1\xi_1\qquad \nu_i = \fr{m_i}{m_1+m_2}.$$
This corresponds to  the orthogonal basis
$$e_1 = (-1,\nu_2,\nu_1)\qquad e_2 = (0,1,-1)$$
and we have mass metric
$$G = \m{\mu_1&0\cr 0&\mu_2}\text{ where }
\mu_1 = \fr{m_1 m_2}{m_1+m_2}\qquad \mu_2 =   \fr{(m_1+ m_2)m_3}{m}.$$
The momenta satisfy
$$\eta_1 = -P_{12}+\nu_2 P_{31}+\nu_1 P_{23}\qquad \eta_2 = P_{31}-P_{23}$$
and for an inverse we could choose
$$P_{12}=0\qquad P_{31}= \eta_1+\nu_1\eta_2\qquad P_{23}= \eta_1-\nu_2 \eta_2.$$
From (\ref{eq_parametrizedH})  we get the  equally familiar Hamiltonian
$$H(\xi,\eta) = \fr{|\eta_1|^2}{2\mu_1}+ \fr{|\eta_2|^2}{2\mu_2} - 
  \fr{m_1 m_2}{|\xi_1|}-\fr{m_1 m_3}{|\xi_2 + \nu_2\xi_1|}-\fr{m_2 m_3}{|\xi_2-\nu_1\xi_1|}.$$
\end{example}

\section{Spherical-Homogeneous coordinates}\label{sec_spherical}
The Hamiltonian $H_{rel}(Q,P)$ of (\ref{eq_HamiltonianPQ3bp}), representing the translation-reduced planar three-body problem, has further symmetries.   The potential function $U(Q)$ is symmetric under simultaneous rotation of  the $Q_{ij}$ in $\C$ and is also homogeneous of degree $-1$ with respect to scaling.  In this section we exploit the
scaling symmetry by converting the system to spherical coordinates.  This will be useful later when we blow-up the triple collision singularity.

We use the mass norm $r= |Q|$ as a measure of the size of a configuration $Q= (Q_{12},Q_{31},Q_{23})\in\C^3$.   In particular, $r=0$ represent triple collision.  For $Q\in\C^3_0$ we want to define a spherical variable $X \in\S^5$
to describe the normalized configuration.   However, instead of using the unit sphere 
$\S^5=\{X\in\C^3: |X|=1\}$ we will view the sphere as the quotient space of $\C^3_0$ under scaling by positive real numbers.   This gives a convenient way to work globally on $\S^5$.  We will take a similar approach when working with the complex projective space $\CP^2$ in the next section.

Let $M = T^*\C^3_0 \simeq \C^3_0\times \C^{3*}$ with the standard symplectic structure.
Let $G=\R^+$ be the group of positive real numbers and let $G$ act on $M$ by $k\cdot (X,Y) = (kX, Y/k)$ where
$X\in\C^3_0, Y\in\C^{3*}, k>0$.  We will use the notation $[X], [X,Y]$ to denote equivalence classes under scaling.  In other words, two vectors $X, X'\in\C^3_0$ are equivalent,  denoted $X'\sim X$ if $X' = kX$ for some $k>0$.  Similarly
$(X',Y')\sim (X,Y)$ if $X'= kX, Y' = Y/k$ for some $k>0$.

The momentum map for this group action is given by $S(X,Y) =\re \metric{Y,X}$ where the angle bracket denotes the Hermitian evaluation pairing (\ref{eq_pairing}).  Fixing this scaling-momentum to be
$\re\metric{Y,X} = 0$ and passing to the quotient space we get a reduced symplectic manifold which can be identified with the cotangent bundle $T^*\S^5$.  This is a special case of  cotangent bundle reduction at zero momentum, as described in theorem ~\ref{th_cotangentreduction0}.   Introduce the notation
$$T^*_{sph}\C^3 =S^{-1}(0) =  \{(X,Y)\in T^*\C^3_0:\re\metric{Y,X} = 0\}.$$
Then we have
$$T_{sph}\C^3/\R^+ \simeq T^*\S^5.$$

We are going to pass from the relative configuration variable $Q\in \C^3_0$ to a size variable $r$ and a homogeneous variable $X\in\C^3_0$.  
\begin{definition}  $(r,X) \in \R^+ \times \C^3_0$ are 
spherical-homogeneous coordinates for the configuration $Q \in \C^3_0$ 
provided $r=|Q|$ and $[X]=[Q]$.  
\end{definition}

$X$ will be defined only up to a positive real factor and will be viewed as representing a point of $\S^5$.   We can use $Q$ itself as a homogeneous representative of the corresponding point in $\S^5$.  Hence we define 
a spherical-homogeneous  coordinate map
$$f : \C^3_0\into \R^+ \times \C^3_0 \qquad r= |Q|, \, X= Q.$$
Extend $f(Q)$ to a map $F(Q,P)$, $F:T^*\C^3_0 \into T^*\R^+ \times T^*_{sph}\C^3$ by setting 
$$
F: \qquad p_r = \fr{\re\metric{P,Q}}{|Q|}\qquad Y = P - \fr{\re\metric{P,Q}}{|Q|^2}Q^*$$
Here $p_r\in\R^*, Y\in \C^{3*}$ are the conjugate momentum variables to $r, X$ and $Q^*$ is the dual covector to $Q$ with respect to the mass metric.  By definition, this means the unique covector in $\C^{3*}$ such that
$$\metric{Q^*,V} = \metric{Q,V}$$
where the first angle bracket is the evaluation pairing and the second is the mass metric.  We find
\begin{equation}\label{eq_Qstar}
Q^* = \fr1m(m_1m_2 Q_{12},m_1m_3 Q_{31},m_2m_3 Q_{23}) \in\C^{3*}.
\end{equation}

A pseudo-inverse $G(r,p_r,X,Y)$, $G:T^*\R^+ \times T^*_{sph}\C^3\into T^*\C^3_0$ to $F$  is given by
\begin{equation}\label{eq_sphericalG}
G:\qquad Q = \fr{rX}{|X|}\qquad  P =  \fr{p_r}{|X|}X^* +  \fr{|X|}{r} Y.
\end{equation}
We have $G\circ F = \id$ and
$$F\circ G(r,p_r,X,Y) = (r,p_r, kX, Y/k)\quad\text{where }k= \fr{r}{|X|}.$$
Hence $f\circ G = \id \bmod \,\R^+$.

To check that $F,G$ are partially symplectic, compute the pull-backs of the canonical one-forms 
\begin{equation}\label{eq_canonicaltheta}
\t = p_r\,dr+ \re(\bar Y_{12}\,dX_{12}+\bar Y_{31}\,dX_{31}+\bar Y_{23}\,dX_{23})
\end{equation}
and $\Theta$ from (\ref{eq_canonicaloneforms}).  We find
$G^*\t = \Theta$ while $F^*\Theta = \t +\ldots$ where the omitted terms are divisible by $\re\metric{Y,X}$.  Hence the maps preserve the restricted symplectic forms as required.

The spherical-homogeneous Hamiltonian  is $H_{sph} = H_{rel}\circ G$.  
Using the formula for $Q$ in (\ref{eq_sphericalG}), the potential 
$U(Q)$ becomes  $U_{sph}(r,X) = \fr1r V(X)$ where
\begin{equation}\label{eq_shapepotential}
V(X) = |X|\,U(X) = |X|\left( \fr{m_1m_2}{|X_{12}|} +  \fr{m_3m_1}{|X_{31}|} + \fr{m_2m_3}{|X_{23}|} \right).
\end{equation}
Note that $V$ is invariant with respect to scaling of $X$ so it determines a well-defined function, $V:\S^5\into \R$, 
which we will sometimes write as $V([X])$.  

The kinetic energy is $K_{sph} = K(P)$ where $P$ is given by (\ref{eq_sphericalG}).
It follows from  lemma~\ref{lemma_kinetic} that the two terms in  (\ref{eq_sphericalG}) are orthogonal with respect to the quadratic form $K$.  To see this, note that they are orthogonal with respect to the dual mass metric since
$\metric{Y,X^*} = \metric{Y,X} = 0$.  Since $X\in\cW$ we have 
$$\metric{Y\circ \pi_\cW,X^*\circ \pi_\cW} = \metric{Y\circ \pi_\cW,\pi_\cW X } =  \metric{Y,X}=0$$
so $X^*\circ \pi_\cW, Y\circ \pi_\cW$ are still orthogonal.  Evaluating $K$ separately on the two terms of 
(\ref{eq_sphericalG}) we find
\begin{equation}\label{eq_Kspherical}
K_{sph} =  \fr12 p_r^2 + \fr{|X|^2}{r^2}K(Y)
\end{equation}
and so the spherical-homogeneous Hamiltonian is 
\begin{equation}\label{eq_Hamiltonianspherical}
H_{sph}(r,p_r,X,Y) = \fr12 p_r^2 +\fr{|X|^2}{r^2} K(Y)-  \fr1r V([X]).
\end{equation}

\begin{theorem}\label{th_spherical}
The Hamiltonian flow of $H_{sph}$ on $T^*\R^+\times T^*\C^3_0$  has   invariant submanifold 
$ \{\re\metric{Y,X} = 0 \}$ and  the quotient of the restricted flow   by the scaling symmetry is equivalent to the Hamiltonian flow of $H_{rel}$ on $T^*\C^3_0$.  This submanifold contains a codimension 2   invariant submanifold  
 $\{ \re\metric{Y,X} = 0,  X_{12}+X_{31}+X_{23}=0\}$ for which  the quotient of the restricted flow by the 
 symmetry of scaling  and   translations of the $Y_{ij}$ is conjugate to the  flow of the zero total momentum three-body problem reduced by translations.
\end{theorem}
\begin{proof}
For the first part we apply theorem~\ref{th_tworeductions} with $M_1 = T^*\C^3_0$, $M_2 = T^*\R^+\times T^*\C^3_0$ and symmetry groups $G_1=\{\id\}$ and $G_2 = \R^+$.  The momentum level is $S(X,Y) = \re\metric{Y,X} = 0$.
It was shown above that the maps $F,G$ between $T^*\C^3_0$ and $S^{-1}(0)$ are partially symplectic pseudo-inverses.

For the second part we change the groups to be $G_1=\C^*$ and $G_2$ is a semidirect product of the scaling group $\R^+$ and the momentum translation group $\C^*$ with group multiplication
$$(k_2,c_2)\cdot (k_1,c_1) = (k_2k_1, c_1/k_2 + c_2)$$
where $(k_i,c_i)\in \R^+\times \C^*$. The momentum levels are $\{Q_{tot} = 0\}$ and $\{X_{tot} = 0, \re\metric{Y,X} = 0\}$ respectively, and these are fixed by the actions of the groups.  The maps $F,G$ restrict to maps between these level sets and the restrictions are partially symplectic pseudo-inverses.
\end{proof}

If we use the formula $K(Y) = \fr12 \bar Y^TBY$ with $B$ from (\ref{eq_Bmatrix}) we find
Hamilton's equations for $H_{sph}$  are
\begin{equation}\label{eq_sphericalHamiltonODE}
\begin{aligned}
\dot r &= p_r  \\
\dot p_r &= \fr{2|X|^2\,K(Y)}{r^3} - \fr1{r^2}V(X) \\
\dot X  &= \fr{|X|^2}{r^2}B Y\\
\dot Y &= \fr1{r}DV(X) -\fr{2K(Y)}{r^2}X.
\end{aligned}
\end{equation}

The quotient space of $T^*\R^+\times T^*_{sph}\C^3_0$ mentioned in theorem~\ref{th_spherical}  is diffeomorphic to  $T^*\R^+\times T^*\S^5$ (by simply thinking of $X,Y$ as homogeneous coordinates for $[X,Y]\in T^*\S^5$).    The quotient space of $T^*\R^+\times T^*_{sph,\cW}\C^3_0$ is diffeomorphic to $T^*\R^+\times T^*S(\cW)$ where
$S(\cW) = \cW \intersection \S^5$ is diffeomorphic to $\S^3$.  Hence the reduced space is eight-dimensional as before.
The reduced flow is just the translation-reduced three-body problem in spherical coordinates.  

At this point, instead of reducing the number of dimensions, we have actually increased it from 12 to 14.
The value of the present formulation lies in the fact that it has been put in a form where double collisions can be easily regularized and the triple collision easily blown-up without destroying the symmetry among the masses.  As in the previous section, one could explicitly realize the reduction to eight dimensions by parametrizing the subspace $\cW$. 
However we will not do this here.

\section{Reduction by rotations -- the shape sphere}\label{sec_reduction}
Next we form the quotient by  rotations.  Since we are using complex coordinates, the combined action of  scaling $Q$  by a real factor $r >0$ and rotating $Q$ by an angle $\t$  is  represented  by $Q \mapsto k Q$ where  $k = re^{i\t} \in \C_0 = \C\setminus 0$, the space of nonzero complex number.   A point in the resulting quotient space represents the size and shape of a configuration.
 
 \subsection{Radial-homogeneous coordinates}\label{sec_radialhomogeneous}
As before we will measure the size by
$$r = |Q|.$$ 
To represent the shape we project $Q \in \C^3_0$ to the quotient of  $\C^3_0 $ by the action of $\C_0$.  
This  quotient space is  the complex projective plane $\bP(\C^3) = \CP^2$.    Homogeneous coordinates will provide a way to work globally on the projective plane, just as they did for the sphere $\S^5$ in the last section.   For $X\in\C^3_0$ let $[X]\in \CP^2$  
denote the corresponding element of the projective plane, i.e., the equivalence class  of $X$ under the relation that $X\sim Q$ if $X=k Q$ for some  $k\in\C, k\ne 0$.   (Thus the square bracket will now mean a projective point rather than a spherical one.)
\begin{definition}  $(r,X)$ are a pair of radial-homogeneous coordinates for $Q \in \C^3_0$
if $r=|Q|$ and $[X]=[Q] \in \CP^2$.  
\end{definition}
$X$ is defined only up to a nonzero complex factor.  We can take $X = Q$ itself to   define 
the  radial-homogeneous coordinate map
$$f : \C^3_0\into \R^+ \times \\C^3_0 \qquad r= |Q|, \, X= Q.$$

\begin{remark} 
 Despite the fact that  spherical-homogeneous coordinates
and radial homogeneous coordinates are both denoted $(r,X)$ 
there are differences between the two coordinate systems.   Spherical-homogeneous coordinates
represent points in $\C^3_0\simeq \R^+\times \S^5$ whereas radial homogeneous coordinates 
represent points in the quotient space  $(\C^3_0)/S^1\simeq \R^+\times\CP^2$.   

 If we include the origin and form the quotient space under rotations we have
$\C^3/S^1 = Cone(\CP^2)$, the cone over $\CP^2$,  where the cone point corresponds to total collision
$0 \in \C^3$.
For any topological  space $X$, we can form the space  $Cone(X)$
which has a distinguished cone point $*$ and $Cone(X)  \setminus * = \R^+ \times X$.   In this case, the cone is not a smooth manifold.
\end{remark}

The equivalence class $[X] = [Q]\in\CP^2$ represents the shape of a three-body configuration only if  $Q\in {\cal W}$.
Restricting to such $Q$ we get   $[Q]\in \bP(\cal W)$, 
where $ \bP(\cal W)$ is  the projective space of the subspace $\cal W\subset \C^3$.  Since $\cal W$ is a two-dimensional complex subspace, 
$\bP(\cal W)$ is a projective line, i.e., 
$$\bP(\cal W)\simeq \CP^1 \simeq \S^2.$$
$\bP(\cal W)$ will be called the {\em shape sphere}.

Any function  on our original configuration space which is invariant under translation, rotation, and scaling
induces a function on the shape sphere, the most important example being our homogenized potential
$$V(X) = |X|U(X) :\bP{\cal W} \to \R.$$

We will also use homogeneous momentum variables.  A pair $(X,Y) \in  T^*\C^3_0 \simeq \C^3_0\times \C^{3*}$ will represent a point of $T^*\CP^2$.   Let $G=\C_0$ be the group of nonzero complex numbers and let $G$ act on $T^*\C^3_0$ by $k\cdot (X,Y) = (kX, Y/\bar{k})$.  We will use the notation $[X,Y]$ to denote equivalence classes under scaling.  In other words, $(X',Y')\sim (X,Y)$ if $X'= kX, Y' = Y/\bar{k}$ for some nonzero $k\in\C$.  The momentum map for this group action is given by the Hermitian evaluation pairing $\s(X,Y) = \metric{Y,X}\in\C$.  The real part of  the complex number $\s(X,Y)$ is the real scaling-momentum $S(X,Y)$ (which we want to be zero as in the last section).  On the other hand, from (\ref{eq_realimag}) we see that $\im\s(X,Y) = -i\,\mu$, where $\mu$ is the angular momentum.

If we fix the complex scaling-momentum to be $\metric{Y,X} = 0$ and pass to the quotient space, then as in theorem ~\ref{th_cotangentreduction0}  we get a reduced symplectic manifold which  is naturally identified with the cotangent bundle 
$T^*\CP^2$ with its natural symplectic structure.  Introduce the notation
$$T^*_{pr}\C^3 =\s^{-1}(0) =  \{(X,Y)\in T^*\C^3_0:\metric{Y,X} = 0\}.$$
Then we have
$$T_{pr}^* \C^3/\C_0 \simeq T^*\CP^2.$$
If, on the other hand, we fix the complex scaling-momentum to be $\metric{Y,X} = -i\,\mu$ and pass to the quotient space we still get a reduced symplectic manifold which can be identified with the cotangent bundle $T^*\CP^2$ but with a twisted symplectic structure, as described in theorem~\ref{th_cotangentreductionmu}.   More about this below.

To get a system equivalent to the reduced three-body problem we  will also need to include the radial variables.
Restrict $X$ to $\cW$ and quotient by the action of  the group $\C$ of translations in $Y$-momentum space.  Let  $M =T^*\R^+\times T^*\C^3_0$ with coordinates $(r,p_r,X,Y)$ and let $G=\C_0\times \C$ acting by $(k,c)\cdot (r,p_r,X,Y) = (r,p_r,kX,c\cdot (Y/\bar k))$ where $c\cdot Y = (Y_{12}+c,Y_{31}+c,Y_{23}+c)$.
Fixing the momentum level $J(X,Y) = (\s(X,Y),X_{tot}) = (-i\,\mu,0)\in\C^2$ and passing to the quotient space gives the reduced phase space
$$P = \{(r,p_r,X,Y): \metric{Y,X} = -i\,\mu, X_{12}+X_{31}+X_{23}=0\}/G$$
of real dimension $\dim P = 14-4-4 = 6$ as expected.  In fact we have
$$P\simeq T^*\R^+\times T^*\bP(\cW) \simeq T^*\R^+\times T^*\S^2.$$ 

We still need to find the reduced Hamiltonian and show that the reduced Hamiltonian system is equivalent to the reduced three-body problem.
This is easy to do starting from the spherical Hamiltonian in the last section.   Indeed, the passage from the homogeneous-spherical  variables $(r,p_r,X,Y)\in T^*\R^+\times T^*\C^3_0$ to the corresponding radial-homogeneous ones is just given by the identity map.  The new feature here is that the symmetry group is enlarged from  $\R^+\times \C^*\simeq \R^+\times\C$ to $\C_0 \times\C$.   Then we have the following extension of theorem~\ref{th_spherical}:
\goodbreak
\begin{theorem}\label{th_reduced1}
The Hamiltonian flow of $H_{sph}$ on $T^*\R^+\times T^*\C^3_0$  has an invariant set where $\metric{Y,X} = -i\,\mu$.  The quotient of the restricted flow by the complex scaling symmetry is equivalent to the Hamiltonian flow of $H_{rel}$ on $T^*\C^3_0/\S^1$.  There is another invariant set where $\metric{Y,X} = -i\,\mu$ and $X_{12}+X_{31}+X_{23}=0$ and the quotient of the restricted flow by the complex scaling symmetry and by translations of the $Y_{ij}$ is conjugate to the  flow of the three-body problem with zero total momentum and angular momentum $\mu$, reduced by translations and rotations.
\end{theorem}
\begin{proof}
The maps $F$ and $G$ as in the proof of theorem~\ref{th_spherical} restrict to maps of the $\mu$ angular momentum levels. They are still partially symplectic pseudo-inverses.
\end{proof}

The next step is to use a momentum shift map to pull-back the problem to the zero-angular-momentum level.  This expresses all of the reduced problems on the same phase space and makes the role of the angular  momentum constant explicit.
Let
\begin{equation}\label{eq_momentumshift}
\Phi_\mu(r,p_r,X,Z) = (r,p_r,X,Y) \qquad Y= Z + \mu\Gamma(X) \qquad \Gamma(X) = \fr{i\,X^*}{|X|^2}
\end{equation}
where
$$X^* = \fr1m(m_1m_2 X_{12},m_3m_1X_{31}, m_2m_3X_{23})\in\C^{3*}.$$
Note that  $\Phi_\mu: J^{-1}(0,0)\into J^{-1}(-i\,\mu,0)$ since if $\metric{Z,X} = 0$ we have
$$\im\metric{Y,X}= \im\metric{i\mu\fr{X^*}{|X|^2},X} = -\mu\re \metric{\fr{X^*}{|X|^2},X}=-\mu.$$
Composing $H_{sph}$ with $\Phi_\mu$ we get a Hamiltonian
\begin{equation}\label{eq_Hamiltonianreduced}
H_{\mu}(r,p_r,X,Z) = \fr12(p_r^2+\fr{\mu^2}{r^2}) +\fr{|X|^2}{r^2} K(Z)-  \fr1r V([X]).
\end{equation}

To verify this we need to show that the kinetic energy can be written
\begin{equation}\label{eq_Kmu}
K_\mu = \fr12(p_r^2+\fr{\mu^2}{r^2}) +\fr{|X|^2}{r^2} K(Z).
\end{equation}
This decomposition  follows from  an orthogonality argument based on lemma~\ref{lemma_kinetic}.  Namely, the vectors 
$i\, X$ and $Z$ are orthogonal with respect to the mass metric and the first one lies in $\cW$.  Then, as in the last section, lemma~\ref{lemma_kinetic} shows that they are orthogonal with respect to the quadratic form $K$ and so $K(Y) = K(\mu\Gamma(X))+K(Z)$.  $K(\mu\Gamma(X))$ gives $\mu^2$-term in $K_\mu$.

Equation \ref{eq_Kmu} gives a decomposition of the kinetic energy into radial and  angular parts and a third term which can be viewed as the kinetic energy due to changes in the shape of the configuration. 
Some authors call this decomposition of kinetic energy, or the consequent orthogonal decomposition of velocities
the ``Saari decomposition''. (See \cite{Saari}.) In the next subsection we show how this last shape term can be understood in terms of the Fubini-Study metric on the shape sphere.

\subsection{Fubini-Study Metrics and the Shape Kinetic Energy}\label{sec_FS}
Using a complex orthogonal basis, we  give a simple decomposition of the dual mass metric which leads to deeper insights into the kinetic energy decomposition  (\ref{eq_Kmu}).  Since the shape sphere has complex dimension one, there are some very simple formulas for the shape term of this decomposition. 

To describe the Fubini-Study metric (also called the K\"ahler metric),  let $\cV$ denote any complex vector space and let $\metric{V,W}$ be any Hermitian metric on $\cV$.  If $X\in\cV_0 = \cV\setminus 0$ then the corresponding {\em Fubini-Study} metric on $T_X\cV$ is:
\begin{equation}\label{eq_FSmetric}
\metric{V,W}_{FS} = \fr {\metric{V,W}\metric{X,X}-\metric{V,X}\metric{X,W}}{\metric{X,X}^2}.
\end{equation}
As a bilinear form on $T_X \cV$, the Fubini-Study ``metric'' is degenerate with kernel the complex line spanned by the vector $X$.   But it induces a bona fide Hermitian metric on the projective space $\bP(\cV)$. 

To see this, let  $\pi:\cV_0\into \bP(\cV)$ denote the projection map: $\pi(X) =  [X]$. The tangent map
$T\pi: T\cV_0\into T\bP(\cV)$, $T\pi(X,V) = ([X],D\pi(X)V)$ has the property that
$T\pi(X,V) = T\pi(X',V')$ if and only if $X' = kX$ and $V' = kV+lX$ for some complex numbers $k\ne0, l$.  So it is natural to view the tangent bundle $T\bP(\cV)$ as the set of equivalence classes of $[X,V]$ of pairs $(X,V)\in\cV_0\times \cV$ under this equivalence relation.   It is easy to check that the formula for $\metric{.,.}_{FS}$ is invariant under this equivalence relation and so it gives a well-defined Hermitian metric on $\bP(\cV)$.  The real part $\re \metric{V,W}_{FS}$ gives a Riemannian metric on $\bP(\cW)$ and the imaginary part gives a two-form called the {\em Fubini-Study form} which will be important later
$$\Omega_{FS}(V,W) = \im\metric{V,W}_{FS}.$$

Starting with the mass metric on $\cV = \C^3$ we get a Fubini-Study metric on $\CP^2$.  However, because of lemma~\ref{lemma_kinetic}, we will be interested in its restriction to the two-dimensional complex subspace $\cW\subset \C^3$ which we denote by $\metric{.,.}_{FS,\cW}$, which induces a Hermitian metric on the shape sphere $\bP(\cW)$.   

Our goal is to show that the shape kinetic energy is the cometric   dual  to this Fubini-Study metric  on $\bP(\cW)$.
(By a `cometric' on a manifold $X$ we mean the fiberwise quadratic form on  $T^*X$   which is  dual to a Riemannian metric on $X$.)  To   this  end we  will  need to  describe cometrics  on projective   space in homogeneous coordinates.  
We continue to identify  $T^*\CP^2$ with the quotient space of  $T^*_{pr}\C^3=\{(X,Z)\in\C^3_0\times \C^{3*}:\pairing{Z,X} = 0\}$ under the complex scaling symmetry. 
 In the same spirit, the cotangent bundle $T^*\bP(\cW)$ is the quotient space ( a symplectic reduced space)
$$T^*\bP(\cW)\simeq (T^*_{pr,\cW}\C^3 _0)/\C_0\times\C$$
where
$$T^*_{pr,\cW}\C^3 _0=\{(X,Z)\in\cW \times \C^{3*}:\pairing{Z,X} = 0, X \ne 0\}$$
and where the group $\C_0\times\C$ represents the scaling symmetry and the momentum translation in $Z$-space.   
We refer to $(X,Z)$ as homogeneous coordinates   on $\bP(\cW)$.  The restriction of $Z \in   \C^{3*}$
to $\cW$ representing a covector in $T^* _{[X]} \bP(\cW)$.  
Expressed in  homogeneous coordinates  a cometric  on $\bP(\cW)$  is a function of the form 
  $Q(X,Z)$ which is quadratic  in $Z$  and invariant under the  $\C_0\times\C$ action. 
\begin{theorem}\label{thm_FSthm}
The Fubini-Study cometric $\norm{Z}^2_{FS,\cW}$ at $[X] \in \bP \cW$ is related to the kinetic energy (formula (\ref{eq_kinetic}))
by  $$\fr{1}{2} \norm{Z}^2_{FS,\cW}  =  |X|^2 K(Z)$$
\end{theorem}
\begin{proof}
Substitute $(X,Z)$ for $(Q,P)$ 
in formula (\ref{eq_kinetic}).  Use $\metric{Z,X} =0$ to get
$K(Z) = \fr{1}{2 |T|^2}\metric{Z,T}$.  The  vector field  $T(X)$ appearing in that formula   is tangent to $\cW$ and orthogonal to $X$, 
hence fits the hypothesis of  lemma \ref{lemma_FSlemma} immediately below.  The lemma asserts that with 
$e(X) = \fr{|X|}{|T(X)|}T(X)$ we have   
$\norm{Z}^2_{FS,\cW} = |\pairing{Z,e (X)}|^2$. 
\end{proof}

\begin{lemma}\label{lemma_FSlemma}
Let $T(X), X\in\cW_0$ be a nonzero complex vectorfield tangent to $\cW _0$ and normal to $X$ with respect to the Hermitian metric mass metric.   Then  $e(X) = \fr{|X|}{|T(X)|}T(X)$
 is a {\em unit} tangent vectorfield on $\cW _0$ with respect to the pulled back Fubini-Study metric $\metric{.,.}_{FS,\cW}$.
 Moreover
 \begin{equation}\label{eq_FSwithT}
\metric{V,W}_{FS,\cW}  = \fr{\metric{V,e(X)}\metric{e(X),W}}{|X|^4}\qquad V,W\in \cW/(\C X) \cong T_{[X]} \bP (\cW) , 
\end{equation}
and the pulled-back cometric is given by the  quadratic form
 \begin{equation}\label{eq_dualFSwithT}
\norm{Z}^2_{FS,\cW} = |\pairing{Z,e (X)}|^2\qquad Z\in T^*_{X,pr}\C^3.
\end{equation} 
\end{lemma}
\begin{proof}
Since $T(X)$ is orthogonal to $X$,  (\ref{eq_FSmetric}) gives
$$|T|^2_{FS} = \fr{|T|^2}{|X|^2}$$
and so $e(X)$ is a Fubini-Study unit vector at $X$.  

The tangent space $T_X\cW$ has complex dimension two and $\{X,e (X)\}$ is a basis.  If we expand $V\in T_X\cW$ as
$$V= \fr{\metric{V,X}}{|X|^2}X + \fr{\metric{V,T(X)}}{|T(X)|^2}T(X)$$
and similarly for $W$, then since $X$ is in the kernel  of $\metric{.,.}_{FS}$  we get
$$\metric{V,W}_{FS,\cW} = \metric{V,W}_{FS}  = \fr{\metric{V,T(X)}\metric{T(X),W}}{|X|^2|T(X)|^2} = \fr{\metric{V,e(X)}\metric{e(X),W}}{|X|^4}$$
as claimed.

Observe that  if ${\mathbb E}, \metric{ \cdot, \cdot} $ is a one-dimensional 
complex Hermitian vector space with unit vector $e$
then the cometric on ${\mathbb E}^*$ is  given by the quadratic form $Z \in {\mathbb E}^*  \mapsto |\pairing{Z,e}|^2$.
From this observation the last formula of the lemma follows.
\end{proof}

\begin{remark}  The manifold $\bP (\cW)$, being a two-sphere, admits no non-vanishing vector field.
So how did we just construct a unit vector field  $e(X)$ to  this two-sphere?  We did not! 
The gadget $e(X)$  is a unit section
of the pull-back  $f^* T \bP (\cW)$ of this tangent bundle by the homogenization map $f: \cW_0 \to \bP (\cW)$
which sends $X \to [X]$. This pull-back bundle can be viewed as  a sub-bundle of $T \cW_0$, and
hence $e(X)$ is a vector field on $\cW_0$.
\end{remark}

Using the vector field $T(X)$ of formula (\ref{eq_kinetic}) (with $X$ substituted for $Q$,
we obtain the  Fubini-Study unit tangent vector 
$$e(X) = \sqrt{\fr{m_1m_2m_3}{m}}\,\left(\fr{\bar X_{31}}{m_2}-\fr{\bar X_{23}}{m_1},\fr{\bar X_{23}}{m_1}-\fr{\bar X_{12}}{m_3},\fr{\bar X_{12}}{m_3}-\fr{\bar X_{31}}{m_2}\right).$$
From this expression we get   simple formulas for the Fubini-Study metric and two-form on $\cW$:
\begin{equation}\label{eq_FStensorX}
\metric{.,.}_{FS,\cW} =\fr{m_1m_2m_3}{m|X|^4}\,\bar\s \tensor \s\qquad \Omega_{FS,\cW} =\fr{m_1m_2m_3}{m|X|^4}\,\im \bar\s \tensor \s
\end{equation}
where the complex-valued one-form $\sigma$ is given by any of the following formulas
\begin{equation}\label{eq_sigma}
\sigma =\metric{e,dX} = X_{31}dX_{12}- X_{12}dX_{31} =X_{12}dX_{23}- X_{23}dX_{12}=X_{23}dX_{31}-X_{31}dX_{23}.
\end{equation}
For example, the second formula for $\s$ is obtained by eliminating  $X_{23}, dX_{23}$ from $\metric{e,dX}$ using the equations $X_{23} = -X_{12}-X_{31}$ and $dX_{23} = -dX_{12}-dX_{31}.$
Note that the formulas for $\s$ are independent of the masses.  This implies that the Fubini-Study metrics for different masses are all conformal to one another. 

Similarly we get a formula for the dual norm and the shape kinetic energy:
\begin{equation}\label{eq_FSdualalpha}
|X|^2K(Z) = \fr12\norm{Z}^2_{FS,\cW} =\fr{m |\a(Z)|^2}{2m_1m_2m_3}
\end{equation}
where $\a(Z)$ is given by any of the following formulas
\begin{equation}\label{eq_alpha}
\begin{aligned}
\a &=\fr1m\left(m_1m_2X_{12}(Z_{23}-Z_{31})+m_3m_1X_{31}(Z_{12}-Z_{23})+m_2m_3X_{23}(Z_{31}-Z_{12})\right)\\
&= \fr{|X|^2(Z_{31}-Z_{12})}{\bar X_{23}}=  \fr{|X|^2(Z_{12}-Z_{23})}{\bar X_{31}}=  \fr{|X|^2(Z_{23}-Z_{31})}{\bar X_{12}}.
\end{aligned}
\end{equation}

Our identification of the shape kinetic energy with the  Fubini-Study cometric gives an alternative formula for the reduced Hamiltonian
on $T^*_{pr}\C^3$
\begin{equation}\label{eq_HmuFS}
H_\mu(r,p_r,X,Z) =  \fr12(p_r^2+\fr{\mu^2}{r^2})+\fr{1}{2r^2}\norm{Z}^2_{FS,\cW} - \fr1r V(X).
\end{equation}
where $\norm{Z}^2_{FS,\cW}$ is the  Fubini-Study cometric on $\cW$.

\subsection{Induced symplectic structure and the reduced differential equations}
Using the momentum shift map, we have pulled back the Hamiltonian to the reduced Hamiltonian $H_\mu$ defined on the zero-angular momentum level $T^*\R^+\times T^*_{pr}\C^3$ where
$$T^*_{pr}\C^3 =\{(X,Z)\in T^*\C^{3}: \metric{Z,X}=0\}.$$ 
However, as described in theorem~\ref{th_cotangentreductionmu}, there is also an induced symplectic structure on this set which different from the restriction of the standard one. 
The pull-back of the canonical one form $\theta$ under the momentum shift map (\ref{eq_momentumshift}) is
$$\Phi_{\mu}^*\theta = p_r\,dr + \re \metric{Z,dX} + \fr{\mu}{|X|^2}\im\metric{X^*,dX} = \Theta + \mu\Theta_1$$
with
$$\Theta_1 = \im\fr{ \metric{X^*,dX}}{|X|^2}= \im\fr{\metric{X,dX}}{|X|^2}$$
where we changed the evaluation pairing to the mass metric in the second equation.
The modified symplectic form will be
$$\Omega_\mu = \Omega - \mu d\Theta_1$$
where we find
\begin{equation}\label{eq_FSprime}
d\Theta_1 = 2\im\fr{\metric{dX,dX}|X|^2-\metric{dX,X}\metric{X,dX}}{|X|^4} = 2\Omega'_{FS}
\end{equation}
where $\Omega'_{FS}$ is the Fubini-Study two-form determined by the mass metric on $\C^3$ (as opposed to its restriction to $\cW$ as in section~\ref{sec_FS}.   Geometrically, $\Omega'_{FS}$ represent the curvature of the circle bundle $\S^5\into \CP^2$.

Once we have $\Omega_\mu$ we calculate Hamilton's differential equations using the defining equation for Hamiltonian vectorfields:
\begin{equation}\label{eq_generalHamode}
(\dot r,\dot p_r,\dot X,\dot Z)\interl \Omega_\mu = dH_\mu.
\end{equation}
The interior product with the standard form gives the usual result:
$$(\dot r,\dot p_r,\dot X,\dot Z)\interl \Omega = -\dot p_r\,dr+  \dot r\,dp_r -\re \langle \dot Z,dX\rangle + \re\langle \dot X, dZ \rangle.$$
Since $\Omega'_{FS}$ involves only  $dX$, it can be viewed as a two-form on $C^3$ instead of on phase space.  Moreover, it only affects the differential equations for $\dot  Z$.  Hamilton's equations read:
\begin{equation}\label{eq_reducedODE1}
\begin{aligned} 
\dot r &=  H_{\mu,p_r} \\
\dot p_r &= -H_{\mu,r}\\
\dot X  &= H_{\mu,Z}\\
\dot Z &= -H_{\mu,X} - 2\mu H_{\mu,Z} \interl  \Omega'_{FS}
\end{aligned}
\end{equation}
where $H_\mu$ is given by (\ref{eq_Hamiltonianreduced}).   The term involving the Fubini-Study metric will be called the {\em curvature term}, $T'_{curv}= -2\mu H_{\mu,Z} \interl  \Omega'_{FS}$.

\begin{lemma}\label{lemma_curvatureterm}
If $X\in\cW$ and $\pairing{Z,X}=0$, then the vector $H_{\mu,Z}$ is in $\cW$ and $\metric{X,H_{\mu,Z}} = 0$.  In fact
\begin{equation}\label{eq_Hmuz}
H_{\mu,Z}  = \fr{\overline{\pairing{Z,e}}}{r^2}\,e \in\cW
\end{equation}
where $e(X)$ is as in lemma~\ref{lemma_FSlemma}.

The curvature term $T'_{curv}$ is equivalent under the translation symmetry in $\C^{3*}$ to 
\begin{equation}\label{eq_Tcurv}
T_{curv} =  -\fr{2\mu}{r^2}\,iZ.
\end{equation}
\end{lemma}
\begin{proof}
From (\ref{eq_Hamiltonianreduced}) we have $H_{\mu,Z} = \fr{|X|^2}{r^2}DK(Z)$.  Note that since $Z\in\C^{3*}$, we have $DK(Z):\C^{3*}\into\R$. By duality we can view $DK(Z)$ as a vector in $\C^3$.  Let $X\in\cW$.  Since $\dot X = H_{\mu,Z}$ and $\cW$ is invariant, we must have $H_{\mu,Z}\in\cW$.   If $\pairing{Z,X}=0$ then an orthogonality argument as above shows
$K(Z+X^*) = K(Z)+K(X^*)$ which implies, since $K$ is a quadratic form, that $DK(Z)(X^*) =\metric{DK(Z),X}=0$ as required.

In section~\ref{sec_FS} we showed that in the subspace $\{Z:\pairing{Z,X}=0\}$ we have
$|X|^2K(Z) =  \fr1{2}|\pairing{Z,e}|^2.$
In fact, we will see that the $Z$-derivatives of these two functions also agree:
\begin{equation}\label{eq_equalderivs}
|X|^2 DK(Z) =\overline{\pairing{Z,e}}\,e.
\end{equation}
To see that (\ref{eq_equalderivs}) indeed holds, note that differentiation along the subspace shows that they must agree when evaluated on any  $\delta Z$ with $\pairing{\delta Z,X}=0$.  On the other hand, both sides vanish on the complementary covector $Z' = X^*$.
Note that the right hand side was calculated, as always, by converting to real variables, finding the real derivative and then converting back to a complex vector.  Equivalently, we expand
$$\fr1{2}|\pairing{Z+\delta Z,e}|^2 = \fr1{2}|\pairing{Z,e}|^2 + \re\pairing{\delta Z,\overline{\pairing{Z,e}}\,e}+\ldots$$
for all $\delta Z$, showing that the vector in question is the complex representative of the real vector derivative.

To show the equivalence of $T'_{curv}$ and $T_{curv}$ we will show that they agree when restricted to $\cW$.  
The argument can be based on a kind of Fubini-Study duality.  Namely, if  $V\in\cW$ we will show that
\begin{equation}\label{eq_derivFSdual}
\metric{H_{\mu,Z},V}_{FS} = \fr1{r^2}\pairing{Z,V}
\end{equation}
which means that  $r^2 H_{\mu,Z}$ is a dual vector to $Z$ with respect to the Fubini-Study metric on $\cW$.  To see this note that (\ref{eq_equalderivs}) gives
$$\metric{H_{\mu,Z},V}_{FS} = \fr1{r^2}\fr{\metric{\overline{\pairing{Z,e}}\,e,V}}{|X|^2}=\fr{\pairing{Z,e}\metric{e,V}}{r^2|X|^2}.$$
On the other hand any $V\in\cW$ is  linear combination
$$V=\fr{\metric{X,V}}{|X|^2}X + \fr{\metric{e,V}}{|e|^2}e.$$
Since $e$ is a Fubini-Study unit vector, we have $|e|=|X|$ and so 
$$\fr1{r^2}\pairing{Z,V}=  \fr{\pairing{Z,e}\metric{e,V}}{r^2|e|^2}= \fr{\pairing{Z,e}\metric{e,V}}{r^2|X|^2}$$
and (\ref{eq_derivFSdual}) holds.  From this we can calculate that for any $V\in\cW$
$$T'_{curv}(V) = -2\mu \im\metric{H_{\mu,Z},V}_{FS} =   -\fr{2\mu}{r^2}\im\metric{Z,V}=-\fr{2\mu}{r^2}\re\metric{i\,Z,V}.$$
This means that $T'_{curv}$ and $T_{curv}$ agree as real-valued one-forms on $\cW$ as claimed.  Replacing $T'_{curv}$ by $T_{curv}$ introduces only an irrelevant translation of the momentum $Z$.
\end{proof}

Taking this lemma into account we finally get Hamilton's equations for the reduced Hamiltonian in the form
\begin{equation}\label{eq_reducedHamiltonODE}
\begin{aligned} 
\dot r &= p_r  \\
\dot p_r &= \fr{\mu^2+|X|^2\,2K(Z)}{r^3} - \fr1{r^2}V(X) \\
\dot X  &= \fr{|X|^2}{r^2}DK(Z)\\
\dot Z &= \fr1{r}DV(X) -\fr{2K(Z)}{r^2}X - \fr{2\mu}{r^2}iZ.
\end{aligned}
\end{equation}

Applying theorem~\ref{th_tworeductions} to the momentum shift map and remembering theorem~\ref{th_reduced1} we have:
\begin{theorem}\label{th_reduced2}
The Hamiltonian flow of $H_{\mu}$ on $T^*\R^+\times T^*\C^3_0$  has an invariant set $T^*\R^+\times T^*_{pr}\C^3$ where $\metric{Z,X} = 0$ with symplectic structure given by the restriction of the standard form minus $2\mu\Omega_{FS}$.  The quotient of the restricted flow by the complex scaling symmetry is equivalent to the Hamiltonian flow of $H$ on $T^*\C^3_0/\S^1$.  There is another invariant set $T^*\R^+\times T^*_{pr,\cW}\C^3$ where $\metric{Z,X} = 0$ and $X_{12}+X_{31}+X_{23}=0$ and the quotient of the restricted flow by the complex scaling symmetry and by translations of the $Z_{ij}$ is conjugate to the  flow of the three-body problem with zero total momentum and angular momentum $\mu$, reduced by translations and rotations.
\end{theorem}

This Hamiltonian system represents the reduced three-body problem in a way which is convenient for regularization of binary collisions and blow-up of triple collision.  However, the phase space is still 14-dimensional.  Next we describe how to find lower-dimensional representations of the reduced three-body problem by parametrizing the shape sphere in various ways.

\subsection{Parametrizing the Shape Sphere}\label{sec_parametrizingPW}
The shape sphere is the projective space $\bP(\cW)$.  As in section~\ref{sec_parametrizingW}, choosing a complex basis $\{e_1,e_2\}$ for $\cW$  gives a map $f:\C^2\into\cW$, $X = f(\xi)$.  By viewing $X\in\cW$ and $\xi\in\C^2$ as homogeneous coordinates  we get an induced parametrization of the shape sphere $f_{pr}:\CP^1\into \bP(\cW)$. 

The formulas of section~\ref{sec_parametrizingW} (with $(Q,P)$ replaced by $(X,Z)$) allow us to find the reduced Hamiltonian for any such basis.  If 
$$e_1= (a_{12},a_{31},a_{23})\qquad e_2= (b_{12},b_{31},b_{23}) \in\cW$$
then we have, as before,
$$X_{ij} = \xi_1\,a_{ij} + \xi_2\, b_{ij}$$
and
 $$\bar\eta_1 = \pairing{Y,e_1}\qquad \bar\eta_2 = \pairing{Y,e_2}.$$
We define a Hermitian mass metric and dual mass metric for $\xi,\eta$ to be the pull-backs of the metrics for $X,Y$.  The squared norms are
 $$|\xi|^2= \bar\xi^T\, G\,\xi\qquad |\eta|^2 = \bar\eta^T\,G^{-1}\eta$$
 where $G$ is the matrix with entries $G_{ij} = \metric{e_i,e_j}$, and these squared norms
 represent the mass metric and cometric on   $\cW$. 

The  relation between the cometric and kinetic energy 
yields the Hamiltonian. (See equations  (\ref{eq_Hamiltonianreduced},\ref{eq_Kmu}) and theorem \ref{thm_FSthm}.)\begin{equation}\label{eq_parametrizedHmu}
H_\mu(r,p_r,\xi,\eta) =\fr12\left( p_r^2 +\fr{\mu^2}{r^2}+\fr{|\xi|^2 | \eta| ^2}{r^2}\right)- \fr1r V(\xi)
\end{equation}
 where  the shape potential is
 $$V(\xi) = |\xi|\left( \fr{m_1 m_2}{\rho_{12}}+\fr{m_1 m_3}{\rho_{31}}+\fr{m_2 m_3}{\rho_{23}}\right)\qquad \rho_{ij} = |X_{ij}| = |a_{ij}\xi_1 + b_{ij} \xi_2|.$$
 
To make the map $F$ of section~\ref{sec_parametrizingW} be partially symplectic we need to alter the standard symplectic form in $(\xi,\eta)$-space by subtracting $2\mu\,F^*\Omega'_{FS}$.  Pulling back the Fubini-Study metric 
$\metric{.,.}_{FS}$ by $f$ gives the  Fubini-Study metric in $\xi$ space 
$$\metric{.,.}_{FS} = \fr {\metric{d\xi,d\xi}\metric{\xi,\xi}-\metric{d\xi,\xi}\metric{\xi,d\xi}}{\metric{\xi,\xi}^2}.$$
With the help of (\ref{eq_FStensorX}) one can show
$$\metric{.,.}_{FS} = \fr{g}{|\xi|^4}\,\bar{\s_0}\tensor \s_0  \text{ where }\s_0 = \xi_1d\xi_2 - \xi_2 d\xi_1\qquad g = \deter{G}$$
The Fubini-Study two-form is the imaginary part.

Since $\s_0$ is independent of the choice of basis, the Fubini-Study metrics for various choices of basis are all conformal to one another.  If we choose an orthonormal basis the metrics are Euclidean.   The Fubini-Study metric for a general basis is related to the Euclidean one by 
$$\metric{.,.}_{FS} =\k(\xi)\,\metric{.,.}_{FS,euc}$$
where the conformal factor is
\begin{equation}\label{eq_kappa2}
\k(\xi) =\fr{g|\xi|_{euc}^4}{|\xi|^4}
\end{equation}
where $|\xi|^2_{euc} = |\xi_1|^2+|\xi_2|^2$.

The curvature term can be calculated directly from the definition $H_{\mu,\eta}\interl \Omega_{FS}$ and we find
$$T_{curv} = -\fr{2\mu}{r^2}\,i\eta.$$
Hamilton's equations in $T^*\R^+\times T^*_{pr}\C^2$ are
\begin{equation}\label{eq_parametrizedHmuODE}
\begin{aligned}
\dot r &= p_r  \\
\dot p_r &= \fr{\mu^2+|\xi|^2 |\eta|^2 }{r^3} - \fr1{r^2}V(\xi) \\
\dot \xi   &= \fr{|\xi|^2}{r^2}G^{-1}\eta \\
\dot \eta  &= \fr1{r}DV(\xi) -\fr{ |\eta|^2}{r^2}G\,\xi - \fr{2\mu}{r^2}i\eta.
\end{aligned}
\end{equation}

There  are still 10 variables but the invariant set $T^*\R^+\times T^*_{pr}\C^2$ with $\metric{\eta,\xi}= 0$ is 8-dimensional  and we have a complex scaling symmetry.    The introduction of an affine coordinate on the projective line
yields a full {\em local} reduction to 6 variables.  For example, consider those points  $[\xi] = [\xi_1,\xi_2]\in \CP^1$ with  
$\xi_1\ne 0$.  If $\rho$ is any nonzero constant complex number then every such point has a unique representative of the form
$$[\xi_1,\xi_2] = [\rho, z]\qquad  z = x+i\,y\in\C, $$
thus  parametrizing almost of the shape sphere by a single complex variable $z$, the {\em affine coordinate}.  Of course the roles of $\xi_1,\xi_2$ could be reversed to parametrize the subset with  $\xi_2\ne 0$.   

If $\zeta =  \alpha+ i\beta\in \C^*$ denotes the momentum vector dual to $z$ then the unique extension of
$f(z) = (\rho, z)$ to a partially symplectic map $T^*\C\into  T^*_{pr}\C^2 = \{\metric{\eta,\xi} = 0\}$ is defined by  
$$\xi_1=\rho\qquad \xi_2 = z \qquad \eta_1=-\overline{z}\zeta/\rho \qquad \eta_2 = \zeta.$$
 One computes  the mass metric is  
 $$|\xi(z)|^2 = g_{11}|\rho|^2 + g_{22}|z|^2 + 2\re(\bar \rho g_{12} z)$$
 and the cometric is  
 $$|\zeta|^2 = \fr{|\xi(z)|^2 |\zeta|^2}{g|\rho|^2} \qquad g = det(G_{ij}). $$
 This gives a Hamiltonian system with 6 degrees of freedom
\begin{equation}\label{eq_affineHmu}
H_\mu(r,p_r,x,y,\a,\b) = \fr12\left(p_r^2 +\fr{\mu^2}{r^2}+\fr{|\xi(z)|^4|\zeta|^2}{g |\rho|^2 r^2}\right) - 
\fr{1}{r}V(x,y)
\end{equation}
 where  $$V(z) = |\xi(z)|\left( \fr{m_1 m_2}{\rho_{12}}+\fr{m_1 m_3}{\rho_{31}}+\fr{m_2 m_3}{\rho_{23}}\right)\qquad \rho_{ij} = |a_{ij} + b_{ij} z|.$$
The Fubini-Study form is
 $$\Omega_{FS} = \fr{g}{ |\rho|^2|\xi(z)|^2}\im d\bar z\tensor dz = \fr{g\,dx \wedge dy}{ |\rho|^2|\xi(z)|^2}.$$
The curvature term is just 
$$T_{curv} = -\fr{2\mu}{r^2} i\z$$
as usual.

\begin{example} [Projective Jacobi Coordinates]\label{ex_projectivejacobi}
As a first example, consider using Jacobi coordinates as in section~\ref{sec_parametrizingW}, only this time applied to the homogeneous variables $X,Z$.  As before, the  basis which defines the Jacobi coordinates is the orthogonal basis
$$e_1 = (-1,\nu_2,\nu_1)\qquad e_2 = (0,1,-1).$$
We have
$$X = (-\xi_1, \xi_2+\nu_2\xi_1,-\xi_2+\nu_1\xi_1)\qquad \xi =(-X_{12}, \nu_1X_{31}-\nu_2 X_{23})$$
and 
$$Z = (0, \eta_1+\nu_1\eta_2,\eta_1-\nu_2\eta_2)\qquad \eta = (-Z_{12}+\nu_2 Z_{31}+\nu_1 Z_{23}, Z_{31}-Z_{23})$$
where, as usual, $Z$ is non-unique.

The Hamiltonian is (\ref{eq_parametrizedHmu}) where the shape potential is
$$V(\xi) = |\xi|\left( \fr{m_1 m_2}{|\xi_1|}+\fr{m_1 m_3}{|\xi_2 + \nu_2\xi_1|}+\fr{m_2 m_3}{|\xi_2-\nu_1\xi_1|}\right).$$
The mass matrix $G= \diag(\mu_1,\mu_2)$ has determinant $g=\mu_1\mu_2 = \fr{m_1m_2m_3}{m}$ 
and associated  norm and conorm:
$$|\xi|^2 = \mu_1|\xi_1|^2+\mu_2|\xi_2|^2 \qquad  |\eta|^2 = \fr{|\eta_1|^2}{\mu_1}+ \fr{|\eta_2|^2}{\mu_2}.$$
Hamilton's equations with the curvature term are given by (\ref{eq_parametrizedHmuODE}).

If we introduce affine variables by setting $\xi_1= \rho, \xi_2 = z$ as above and if we choose
$\rho = \sqrt{\fr{\mu_2}{\mu_1}}$ 
the mass norm reduces to 
$$|\xi|^2 = \mu_2(1+x^2+y^2)$$
and we get the affine Jacobi Hamiltonian
$$H_\mu(r,p_r,x,y,\a,\b) = \fr12\left(p_r^2 +\fr{\mu^2}{r^2}+\fr{(1+x^2+y^2)^2|\zeta|^2}{r^2}\right) - 
\fr{1}{r}V(x,y)$$

Hamilton's equations with the curvature term are 
\begin{equation}\label{eq_reducedHamiltonODEaffine}
\begin{aligned}
\dot r &= p_r  \\
\dot p_r &= \fr{1}{r^3} [ \mu^2+ (1+x^2+y^2)^2 (\alpha^2+\beta^2) ]- \fr1{r^2}V(\xi) \\
\dot x   &=  \fr{(1+x^2+y^2)^2}{r^2} \alpha  \\
\dot y    &= \fr{(1+x^2+y^2)^2}{r^2}\beta  \\
\dot \alpha  &= \fr{1}{r}V_x (x,y) -\fr{2}{r^2} (1+x^2+y^2)(\alpha^2+\beta^2) x +\fr{2 \mu}{r^2} \beta \\
\dot \beta  &= \fr{1}{r}V_y(x,y)  -\fr{2}{r^2}(1+x^2+y^2)(\alpha^2+\beta^2) y - \fr{2\mu}{r^2} \alpha.
\end{aligned}
\end{equation}
\end{example}

\begin{example}[Equilateral Coordinates]\label{ex_equilateral}
In projective Jacobi  coordinates $(\xi_1,\xi_2)$,  the binary collision points  $b_{12}, b_{13}, b_{23}$
are located at  the projective points 
$$[1,0], [1,-\nu_2], [1,\nu_1]\in\CP^1
$$
while the equilateral triangle configurations (the Lagrange points) are at 
$$[1,\ell_\pm]\in\CP^1$$
 where
$$
\ell_\pm = \fr{m_1-m_2}{2(m_1+m_2)} \pm \fr{\sqrt{3}}{2}\,i = \fr{\nu_1-\nu_2}{2}\pm \fr{\sqrt{3}}{2}\,i .
$$

 Using a M\"obius transformation, we can put  three points anywhere we like on the shape sphere, $\CP^1$.
 Remarkably, it turns out that if we put the binary collisions at the third roots of unity 
 \begin{equation}
 \label{eq_poles}
 [\xi_1,\xi_2] = [1,1], [1,\omega], [1,\bar\omega]\in\CP^1
 \end{equation}
 with $\omega = -\fr{1}{2} + i \fr{\sqrt{3}}{2}$, then the equilateral points are automatically moved to the north and south poles
 $$[1,0], [0,1].$$
 These coordinates were introduced in \cite{MMV}.
 
 These coordinates are obtained by choosing the basis
 $$e_1 =  (1,\omega,\bar\omega)\qquad  e_2 = - \bar e_1 = (-1,-\bar\omega,-\omega) $$
 for $\cW$.  The coordinate change map is $X = \xi_1\,e_1+\xi_2\,e_2$ or
 $$X_{12} = \xi_1-\xi_2\qquad X_{31}= \omega\xi_1-\bar\omega\xi_2\qquad X_{23}= \bar\omega\xi_1-\omega\xi_2.$$
and indeed takes   the roots of unity (eq.  \ref{eq_poles})  to the binary collisions.  
Setting $\xi_2 =0$ we see that $|X_{12}|= |X_{32}| = |X_{23}|$ corresponding
to an equilateral triangle, with the same result if $\xi_1 = 0$.
Thus the coordinate change map sends the poles  $\xi=[1,0],[0,1]$   to the   equilateral triangles.

The mutual distances (of the homogeneous variables) $\rho_{ij} = |X_{ij}|$ which appear in the shape potential are very simple:
$$\rho_{12} = |\xi_1-\xi_2| \qquad \rho_{31}= |\xi_1-\omega\xi_2| \qquad \rho_{23}=\xi_1-\bar\omega\xi_2|.$$
The mass metric can also be written in terms of these
$$|\xi|^2 = \fr1m(m_1m_2\rho_{12}^2+m_3m_1\rho_{31}^2+m_2m_3\rho_{23}^2).$$
It is represented by the matrix $G$ with entries $g_{ij}= \metric{e_1,e_2}$:
$$ g_{11}=g_{22} = \fr{m_1m_2+m_3m_1+m_2m_3}{m} \quad 
g_{12} = \bar g_{21} = -\fr{m_1m_2+m_3m_1\,\omega+m_2m_3\,\bar\omega}{m}$$
and determinant $g = \deter{G} = \fr{3 m_1m_2m_3}{m}$.

The inverse transformation is given by
$$ \xi_1 = \fr13(X_{12}+\bar\omega X_{31}+\omega X_{23})\qquad \xi_2 =  -\fr13(X_{12}+\omega X_{31}+\bar\omega X_{23})$$
and the momenta satisfy
$$\eta_1 = Z_{12}+\bar\omega Z_{31}+\omega Z_{23}\qquad \eta_2 = -Z_{12}-\omega Z_{31} -\bar\omega Z_{23}.$$

Choosing affine variables by setting $\xi_1=z, \xi_2=1$ we get the Hamiltonian (\ref{eq_affineHmu})
with
$$|\xi(z)|^2 =  \fr1m(m_1m_2|z-1|^2+m_3m_1 |z-\omega|^2+m_2m_3|z-\bar\omega|^2).$$
The complexity of  mass norm is perhaps outweighed by the fact that the potential is given by the wonderful expression
$$V(z) =|\xi(z)| \left( \fr{m_1m_2}{|z-1|} +  \fr{m_1m_3}{|z-\, \omega|}+   \fr{m_2m_3}{|z-\, \bar \omega|} \right).$$

The  advantage of these coordinates  is that
   they provide the homogenized potential $V$ with  ``radial monotonicity'''. 
 Let $E = x \dd{}{x} + y\dd{}{y} $ be the radial vector field in the $z$ plane,
      where $z = x + i y$. 
     Then  $E[V] >  0$ for $0 < |z| < 1$,  $E[V] <0$ for $|z| > 0$,
     and  $E[V] = 0$ if and only if $|z| = 1$ or $z = 0$. 
   (See Proposition 4 of  \cite{MMV})  This monotonicity was the  key ingredient to the main theorem of  
    \cite{Mont}.

\end{example}

 \subsection{Making the Shape Sphere Round}\label{sec_shapesphereround}
Instead of using projective or local affine coordinates, one can map the shape sphere to the unit sphere in $\R^3$.   First we do this homogeneously, then restrict to the unit sphere to get another version with 6 degrees of freedom.  Let $\xi=(\xi_1,\xi_2)\in\C^2$ be coordinates associated with some choice of basis $e_1,e_2$ for $\cW$.

Consider the Hopf map  $h:\C^2\into\R^3$ given by
$$w_1=2\re \overline{\xi_1}\xi_2 \qquad w_2=2 \im \overline{\xi_1}\xi_2\qquad w_3= |\xi_1|^2-|\xi_2|^2$$
Using the Euclidian metric for $w$ we get
$$|w|^2 = w_1^2+w_2^2+w_3^2 = |\xi|_{euc}^4 = (|\xi_1|^2+|\xi_2|^2)^2.$$
It follows that
$$2|\xi_1|^2 = |w| + w_3 \qquad 2 |\xi_2|^2 = |w|-w_3 \qquad 2\bar\xi_1\xi_2 = w_1+i\,w_2.$$

We will need formulas for $\rho_{ij} = |X_{ij}| = |a_{ij}\xi_1+b_{ij}\xi_2|$ in the variables $w_i$.  We have
\begin{equation}\label{eq_rhoijw}
\begin{aligned}
\rho_{ij}^2 &= |a_{ij}|^2|\xi_1|^2 +  |b_{ij}|^2|\xi_2|^2 +2\re(\bar\xi_1\xi_2 \bar a_{ij} b_{ij})\\
&= \smfr12(|a_{ij}|^2+|b_{ij}|^2)|w|+\smfr12(|a_{ij}|^2-|b_{ij}|^2)w_3 +\re(\bar a_{ij}b_{ij})w_1-\im(\bar a_{ij}b_{ij})w_2.
\end{aligned}
\end{equation}
Then the mass metric will be given by
\begin{equation}\label{eq_xisqw}
|\xi|^2 = \fr1m(m_1m_2\rho_{12}^2+m_3m_1\rho_{31}^2+m_2m_3\rho_{23}^2).
\end{equation}

If we let $\a_1,\a_2,\a_3$ be dual momentum variables, we can extend the Hopf map $h$ to a partially symplectic map
$F:T^*_{pr}\C^2\into T^*_{sph}\R^3$ by defining its (pseudo) inverse: 
$$\eta = \a \circ Dh : = Dh^t \a.$$
To find the reduced Hamiltonian in $w$ coordinates we will exploit the fact that Euclidean metric transforms nicely.  Recall that the shape kinetic energy is the dual of the Fubini-Study metric and that the latter is related conformally to the Euclidian metric with conformal factor $\k^{-1}$ where
$\k$ is given by (\ref{eq_kappa2}).  In other words, since we are restricting to $\metric{\eta,\xi} = 0$ we have
$$|\xi|^2|\eta|^2 = \k^{-1}|\xi|_{euc}^2 |\eta|_{euc}^2.$$
One can verify that the Euclidean norms transform under the Hopf map in such a way that 
$$|\xi|_{euc}^2 |\eta|_{euc}^2 = 4 |w|^2 |\a|^2$$
where we are using the Euclidean norm on $\R^3,\R^{3*}$.  Hence 
 the reduced Hamiltonian on the sphere is given by:
$$H_\mu(r,p_r,w,\a) = \fr12\left( p_r^2 +\fr{\mu^2}{r^2}+\fr{4|w|^2|\a|^2}{\k(w)r^2} \right) - 
\fr{1}{r}V(w)$$
 where  $|w|^2 =w_1^2+w_2^2+w_3^2$ and $|\a|^2 = \a_1^2+\a_2^2+\a_3^2$
 and where the shape potential is given by
 $$ V(w) = |\xi(w)|  \left( \fr{m_1 m_2}{\rho_{12}}+\fr{m_1 m_2}{\rho_{12}}+\fr{m_1 m_2}{\rho_{12}}\right)$$
 with the $\rho_{ij}$ and $|\xi|$ as in (\ref{eq_rhoijw}) and (\ref{eq_xisqw}).   

The Fubini-Study form becomes a multiple $\k/4$ of the Euclidean solid angle form
 $$\Omega_{FS} = \fr{\kappa}{4|w|^3}(w_1 dw_2\wedge dw_3 +w_2 dw_3\wedge dw_1 +w_3 dw_1\wedge dw_2).$$
 This leads to the curvature term 
 $$T_{curv} = \fr{2 \mu }{|w| r^2} \a\times w$$
  where $w\times\a$ denotes the cross product in $\R^3$.  

The differential equations are:
\begin{equation}\label{eq_reducedHamiltonODEspherical}
\begin{aligned}
\dot r &= p_r  \\
\dot p_r &= \fr{1}{r^3}\left(\mu^2+ \fr{4|w|^2|\a|^2}{\k} \right)- \fr1{r^2}V(\xi) \\
\dot w   &=  \fr{4|w|^2}{\k r^2} \alpha  \\
\dot \alpha  &= \fr{1}{r}DV(w) -\fr{4|\a|^2}{\k r^2} w + \fr{4|w|^2|\a|^2}{\k^2 r^2}\k_w + \fr{2 \mu }{|w| r^2} \a\times w
\end{aligned}
\end{equation}

From theorem~\ref{th_tworeductions}, if we restrict to $T^*\R^+\times T^*_{sph}\R^3 = \{\metric{\a,w}_{euc}=0\}$ and quotient by the scaling action of $\R^+$, we get a reduced system equivalent to the reduced three-body problem.  But $\metric{\a,w}_{euc}=0$ implies that $|w|$ is constant under the flow.   Hence we have a 6-dimensional invariant submanifold given by $|w| = 1,\metric{\a,w}_{euc}=0$ representing the reduced three-body problem.  The reduced phase space is $T^*\R^+\times T^*\S^2$ and the shape sphere is represented by the standard unit sphere.

To get to $6$-dimensions with no constraints one could parametrize the sphere with two variables.  If this is done with stereographic projection, the result is the similar to the affine coordinate reduction of section~\ref{sec_parametrizingPW}.  On the other hand one could also use spherical coordinates $\t,\phi$.  However, both of these are just local coordinates while the system above is global, albeit constrained.

\begin{example}[Jacobi coordinates on $\S^2$]
If we choose an orthonormal basis for $\cW$ then we get the conformal factor $\k=1$ and the resulting Hamiltonian will have a simpler shape kinetic energy.  For example we could normalize the Jacobi basis of example~\ref{ex_projectivejacobi} to
$$e'_1 = \fr{1}{\sqrt{\mu_1}}(-1,\nu_2,\nu_1)\qquad e'_2 = \fr1{\sqrt{\mu_2}}(0,1,-1).$$
The coordinates $\xi_i$ are replaced by $\sqrt{\mu_i}\xi_i$ in all of the formulas.  We get rather complicated homogeneous mutual distances
$$\begin{aligned}
2\mu_1\mu_2 \rho^2_{12} &= \mu_2(|w| + w_3)\\
2\mu_1\mu_2 \rho^2_{31} &=(\mu_2\nu_2^2+\mu_1)|w| + (\mu_2\nu_2^2-\mu_1)w_3+2\nu_2\sqrt{\mu_1\mu_2}w_1\\
2\mu_1\mu_2 \rho^2_{23} &=(\mu_2\nu_1^2+\mu_1)|w| + (\mu_2\nu_1^2-\mu_1)w_3-2\nu_1\sqrt{\mu_1\mu_2}w_1.
\end{aligned}
$$
In the equal mass case
with $m_i=1$ and $|w|=1$, however, we get
$$\rho_{12}^2 = |w|+w_3\qquad \rho_{31}^2 = |w|+\fr{\sqrt{3}}{2}w_1-\fr12 w_3\qquad \rho_{23}^2 = |w|-\fr{\sqrt{3}}{2}w_1-\fr12 w_3.$$

On the other hand the Hamiltonian is
$$H_\mu(r,p_r,w,\a) = \fr12\left( p_r^2 +\fr{\mu^2}{r^2}+\fr{4|w|^2|\a|^2}{r^2} \right) - 
\fr{1}{r}V(w)$$
where the norms are Euclidean.
\end{example}

\begin{example}[Equilateral coordinates on $\S^2$]\label{ex_equilateralsphere}
If we use the basis of example~\ref{ex_equilateral}
 $$e_1 =  (1,\omega,\bar\omega)\qquad  e_2 = - \bar e_1 = (-1,-\bar\omega,-\omega) $$
 we get simple mutual distances
$$\rho_{12}^2 = |w|-w_1 \qquad \rho_{31}^2 = |w|+\fr12 w_1 -\fr{\sqrt{3}}{2}w_2\qquad \rho_{23}^2 = |w|+ \fr12 w_1 +\fr{\sqrt{3}}{2}w_2.$$
Collinear shapes form  the equator $w_3=0$ with the binary collisions placed at the roots of unity.

On the other hand we have a formidable conformal factor
$$
\kappa = \fr{3m_1m_2m_3 m(w_1^2+w_2^2+w_3^2)}{(m_1m_2\rho_{12}^2+m_3m_1\rho_{31}^2+m_2m_3\rho_{23}^2)^2}, 
$$
In  the equal mass case ($m_i=1$ ) we see   $\k=1$.
\end{example}

\subsection{Visualizing the Shape Sphere}
Having reduced the planar three-body problem by using  size and shape coordinates, we will pause to have a closer look at the shape sphere and the shape potential.  

Using the spherical variables $w=(w_1,w_2,w_3)$ we can visualize the shape sphere as the round unit sphere in $\R^3$.  The equilateral basis of example~\ref{ex_equilateralsphere} puts the binary collisions at the third roots of unity on the equator and the Lagrange equilateral configurations at the poles.  Figure~\ref{fig_sphericalV}) shows some of the level curves of $V$ for two choices of the masses. In addition to the binary collisions shapes where $V\into\infty$ , there are three saddle points at the Eulerian central configurations.   The Lagrange points are always minima of $V$.

\begin{figure}[h]
\scalebox{1.2}{\includegraphics{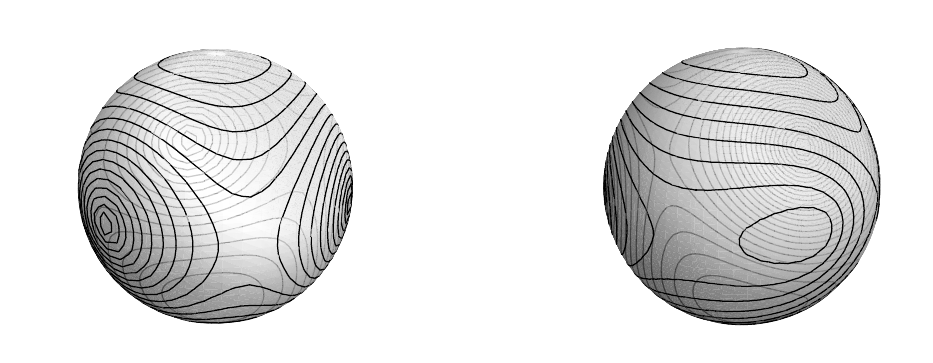}}
\caption{Contour plot of the shape potential on the unit sphere $w_1^2+w_2^2+w_3^2=1$  in the equal mass case (left) and for masses $m_1=1,m_2=2,m_3=10 $ (right).} \label{fig_sphericalV}
\end{figure}

If we use stereographic projection to map the sphere to the complex plane, we get the affine coordinate representation of example~\ref{ex_equilateral}.
Figure~\ref{fig_planarV} shows affine contour plots for the same two choices of the masses.  Now the collinear shapes are on the real axis.

\begin{figure}[h]
\scalebox{0.7}{\includegraphics{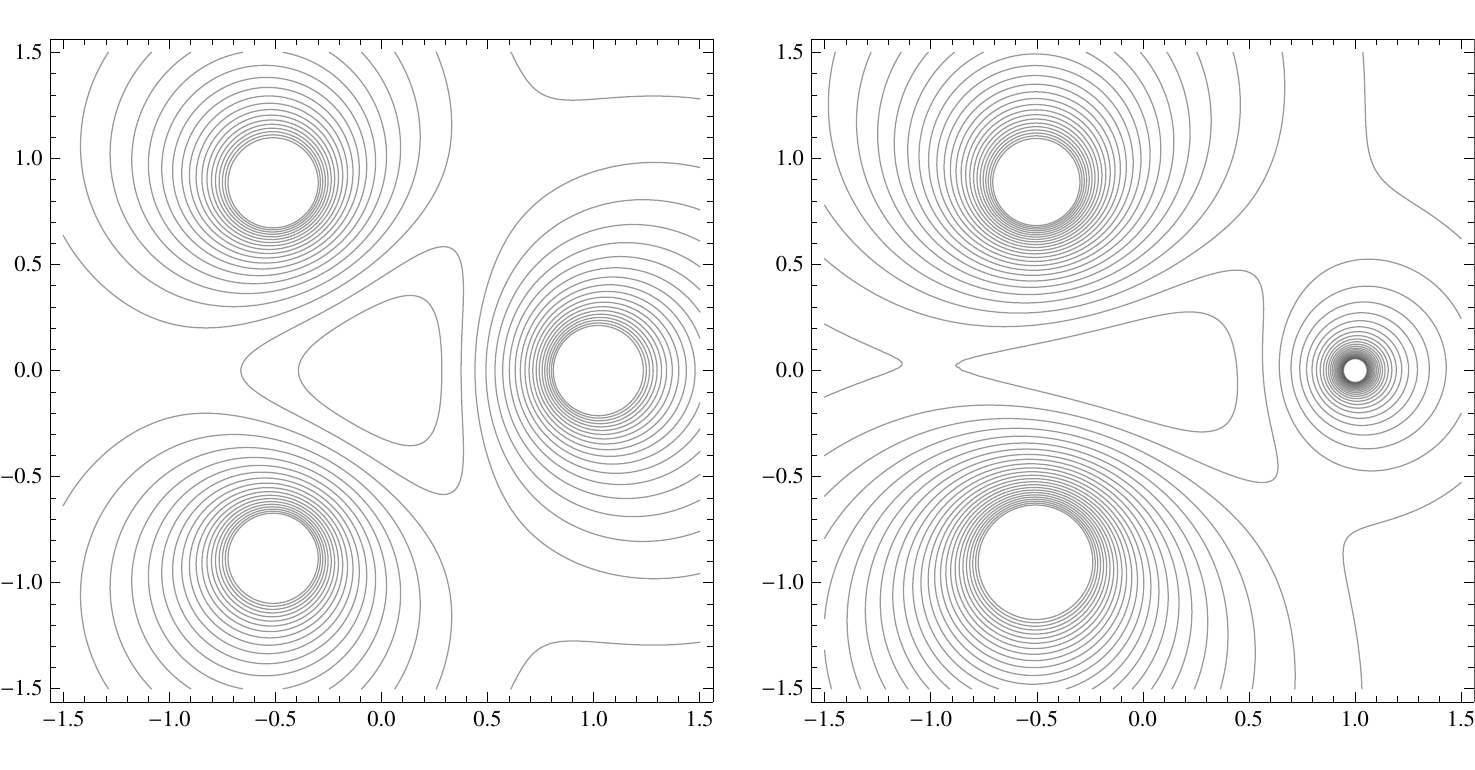}}
\caption{Contour plot of the shape potential on the complex plane in the equal mass case (left) and for masses $m_1=1,m_2=2,m_3=10$ (right).  These plots can be viewed as a stereographic projections of  those in figure~\ref{fig_sphericalV}. }\label{fig_planarV}
\end{figure}

\section{Levi-Civita Regularization} 
In this section, we describe a way to simultaneously regularize all $3$ binary collision using  $3$ separate Levi-Civita transformations.  This approach to simultaneous regularization was introduced by Heggie \cite{Heggie}.  There are two versions depending on whether the variables $Q_{ij}$ or the homogeneous variables $X_{ij}$ are used.  The former approach was used by Heggie; we will take the latter.  We begin with a review of Levi-Civita regularization for the Kepler problem.

Levi-Civita showed how to regularize the two-body problem,
which is to say, the Kepler problem.
Let $q \in \C$ denote the position of a planet going around an
infinitely massive sun placed at   the origin.  After
a normalization, the Kepler Hamiltonian is $\frac12 |p|^2 - \alpha/|q|$.
Levi-Civita's transformation is
 the map
$$z \mapsto z^2 = q$$
together with the induced map on momenta:
$$\eta  \mapsto {\frac12 \bar \eta} = p$$
and the   time rescaling
$${d \over {d \tau}} = r {d \over {dt}}.$$
To understand the map on momenta, make   
  the substitution $q = z^2$ in the expression $\langle p, dz \rangle$
  for the canonical one-form. We 
have $\langle p, dq \rangle = \langle p, 2 z dz \rangle = \langle 2 \bar z p, dz \rangle$
which shows that if  $\eta = 2 \bar z p$ then
$\langle \eta, dz \rangle = \langle p, dq \rangle$. 
This computation shows that  
the map $(\eta, z) \to (p, q)$ with  $p =  {1 \over 2 \bar z} \eta$,  $q = z^2$ is  
 a 2:1   canonical transformation away from the origin.
Observe that $r= |z|^2$.  Thus in terms of the new variables
$$H = \fr{1}{2r} (|\eta|^2 - \fr{\alpha}{|z|^2}).$$

Time rescaling is equivalent to rescaling the Hamiltonian vector field.
This rescaling can be implemented using   the following ``Poincar\'e trick" .
 If $X_H$ is the Hamiltonian vector field for 
$H$, and if $h$ is a value of $H$, then $f X_H$ is the Hamiltonian vector field
for the Hamiltonian $\tilde H = f(H-h)$ {\em provided we restrict ourselves 
to the level set} $\{ H = h \}$. We take $f = r = |z|^2$ and compute that 
$$\tilde H =  \fr{1}{2} (|\eta|^2 - h|z|^2 -  \alpha).$$
which is the Hamiltonian for a harmonic oscillator when $h < 0$.

 \subsection{Simultaneous Regularization}
Let $(r,X)$ denote either the homogeneous-spherical or radial-homogeneous coordinates.
To simultaneously  regularize all three double collisions we perform a Levi-Civita
transformation on each of the homogeneous complex variables $X_{ij}$.
Thus, we introduce three new complex variables $z_{ij} = -z_{ji}$
and set  
$$X_{ij} = z_{ij}^2.$$
Define a {\em regularizing map}  $f: \C^3_0 \into\C^3_0$ by
$$X =f(z_{12}, z_{31}, z_{23}) = (z_{12}^2, z_{31}^2, z_{23}^2).$$
The preimage of the subspace $\cW$ is the quadratic cone $\cC$ with
$$\cW: \qquad  z_{12}^2 + z_{31}^2 + z_{23}^2 = 0$$
and we have $f:\cC_0\into \cW_0$.  Note that every $X\in\cW_0$ has 8 preimages under $f$, except  for the 
three binary collision points ($X_{ij}=0$ some $ij$) which each have 4 preimages. 
(Since $X \ne 0$, at most one of the $X_{ij}$ or $z_{ij}$ can vanish at a time on $\cW_0$ or $\cC_0$.)

Since $f$ is  homogeneous, it induces maps  $f_{sph}:\S^5\into \S^5$ and $f_{pr}:\CP^2\into \CP^2$.
In this case we also view $z_{ij}$ as homogenous spherical or projective coordinates.  These restrict to regularizing maps
 $f_{sph}:\S(\cC)\into \S(\cS)$ and $f_{pr}:\bP(\cC)\into \bP(\cW)$ where, as above, $\S(.), \bP(.)$ denote quotient spaces under real and complex scaling, respectively.

The mutual distances become
\begin{equation}\label{eq_rhoij}
\rho_{ij} = |X_{ij}| = |z_{ij}|^2
\end{equation}
and the mass norm is
\begin{equation}\label{eq_normXsq}
|X(z)|^2 = |f(z)|^2 = \fr{m_1m_2\rho_{12}^2+m_1m_3\rho_{31}^2+m_2m_3\rho_{23}^2}{m_1+m_2+m_3}.
\end{equation}
We will use the standard Hermitian inner product, denoted $\metrictwo{.,.}$, on $z$-space so
\begin{equation}\label{eq_normzsq}
\normtwo{z}^2 = |z_{12}|^2+|z_{31}|^2+|z_{23}|^2 = \rho_{12} +\rho_{31}+\rho_{23}.
\end{equation}

Let  $\eta_{ij}$ be the conjugate momenta to $z_{ij}$ and let $Y_{ij}$ the homogenous momenta conjugate to $X_{ij}$.  We extend $f$ to a map $(r,p_r,X,Y) = F(r,p_r,z,\eta)$ by setting
$$Y_{ij} = {1 \over {2 \bar z_{ij}}} \eta_{ij}.$$
Then $F$ restricts to maps $T^*\R^+\times T^*_{sph}\C^3\into T^*\R^+\times T^*_{sph}\C^3$ and
$T^*\R^+\times T^*_{pr}\C^3\into T^*\R^+\times T^*_{pr}\C^3$ where in $(z,\eta)$-variables we have the constraints 
  $\re\pairing{\eta,z}=0$ for the sphere and $\pairing{\eta,z}=0$ for the projective plane. We continue to denote
  these restricted maps by the letter $F$. 

The   action of $c \in \C$ by translation of the momenta $Y_{ij}$  to $Y_{ij} + c$
pulls-back under $F$ to translation of $\eta_{ij}$ by $2  c \bar z_{ij}$, that is, to the action
$$c\cdot (r,p_r,z,\eta) = (r,p_r,z,\eta + 2 c \bar z).$$
The momentum map  for this pulled back action 
is $\g = z_{12}^2+z_{31}^2+z_{23}^2$.  Of course we will be interested in the level set $\g=0$. 
We will call this the $z$-translation symmetry of $\eta$.

 \subsubsection{Geometry of $\cC$ and the Regularized Shape Sphere}\label{sec_geometryC}
 It is interesting to investigate the algebraic surface $\cC$  in more detail.  If we write
 the complex vector $z\in \C^3$ as $z=a+i\,b$ where $a=\re z, b=\im z \in\R^3$ then
 $$z_{12}^2 + z_{31}^2 + z_{23}^2 = 0\text{ if and only if }|a|^2=|b|^2,\; a\cdot b=0.$$
 This means $a,b$ are real, orthogonal vectors of equal length $s^2 = |a|^2=|b|^2=|z|^2/2$.  If we define a third vector
 $c=a\times b$ we get an orthogonal frame in $\R^3$ and the matrix
 \begin{equation}\label{eq_Aofz}
 A(z) = \fr1s\m{a_{12}&a_{12}&c_{12}/s\\a_{31}&b_{31}&c_{31}/s\\a_{23}&b_{23}&c_{23}/s}\in SO(3).
 \end{equation}
The mapping $A(z)$  induces a diffeomorphism from  the quotient space of $\S(\cC)$ of $\cC_0$ under positive, real scaling  to $SO(3)$ and hence, as is well-known, to the real projective space $\RP(3)$ (and to the unit tangent bundle to $S^2$).

The projective curve $\bP(\cC)$ turns out to be diffeomorphic to the two-sphere $\S^2$ and, accordingly, we will call it the {\em regularized shape sphere}.   One way to see this is to note that  $\bP(\cC)\simeq \S(\cC)/\S^1$ is the quotient of $\S(\cC)$ under rotations.  It is easy to see that  action the rotation group on  $z$ rotates the vectors $a,b\in \R^3$ above in their own plane and leaves $c=a\times b$ invariant.  It follows that the map $z\mapsto c/|c|$ induces a diffeomorphism $\bP(\cC)\simeq \S^2$.

In the sections below, we will apply the regularizing map to obtain several regularized Hamiltonians for the three-body problem.   Starting with homogenous spherical variables leads to a regularized system not reduced by rotations while the radial-homogenous variables lead to a Hamiltonian system which is both regularized and reduced.  In addition we will consider several way to parametrize the cone $\cC$ to obtain lower-dimensional systems.  Theorem~\ref{th_tworeductions} can be applied  to show the equivalence of the Hamiltonian systems below, but we will omit most of the details.

\subsection{Spherical Regularization}\label{sec_sphericalregularization}
First we will find the regularized Hamiltonian in spherical-homogeneous coordinates.  This gives a regularization of binary collisions without reducing by the rotational symmetry.   Let $(r,X)$ be the spherical-homogeneous coordinates of section~\ref{sec_spherical}.  The spherical Hamiltonian is
$$H_{sph}(r,p_r,X,Y) = \fr12 p_r^2 +\fr{|X|^2}{r^2} K(Y)-\fr{1}{r}V(X).$$
Using the formula analogous to the one in (\ref{eq_HamiltonianPQ3bp}) for $K(Y)$ and applying the regularizing map gives
\begin{equation}\label{eq_Hsphxieta}
\begin{aligned}
H_{sph}(r,p_r,z,\eta) = 
\fr12p_r^2 &+ \fr{|X(z)|^2}{r^2}\left(\fr{|\pi_1|^2}{8m_1\rho_{12}\rho_{31}}+\fr{|\pi_2|^2}{8m_2\rho_{12}\rho_{23}}+  \fr{|\pi_3|^2}{8m_3\rho_{31}\rho_{23}}\right)\\ &- \fr{1}{r}\left(\fr{m_1m_2}{\rho_{12}} + \fr{m_3m_1}{\rho_{31}} + \fr{m_2m_3}{\rho_{23}}\right) 
\end{aligned}
\end{equation}
and where
\begin{equation}
\pi_1 =  \eta_{12}\bar{z_{31}}- \eta_{31}\bar{z_{12}}
\quad \pi_2 =\eta_{23}\bar{z_{12}}- \eta_{12}\bar{z_{23}}
 \quad \pi_3 =\eta_{31}\bar{z_{23}}- \eta_{23}\bar{z_{31}}.
 \end{equation}

Next we rescale time using the Poincar\'e trick.  One choice of time-rescaling factor is 
$|z_{12}z_{31}z_{23}|^2=\rho_{12}\rho_{31}\rho_{23}$.  But since $X,z$ are homogeneous coordinates, the degree-zero homogeneous function
\begin{equation}\label{eq_tau}
\tau = \fr{\rho_{12}\rho_{31}\rho_{23}}{(\rho_{12}+\rho_{31}+\rho_{23})^3} = \fr{\rho_{12}\rho_{31}\rho_{23}}{\normtwo{z}^6}
\end{equation}
is more appropriate.  Note that by the arithmetic-geometric mean inequality we have $0\le \tau\le \fr1{27}$.

The rescaled solution with energy $H_{sph}=h$ become the zero-energy solutions of $\tilde H_{sph}(r,p_r,z,\eta) = \tau(H_{sph}-h)$.
\begin{equation}\label{eq_Htildesph}
\tilde H_{sph} = 
\fr{\tau\, p_r^2}{2} + \fr{|X(z)|^2}{r^2\normtwo{z}^6}\left(\fr{|\pi_1|^2\rho_{23}}{8m_1}+\fr{|\pi_2|^2\rho_{31}}{8m_2}+  \fr{|\pi_3|^2\rho_{12}}{8m_3}\right) 
- \fr{1}{r}W(z)  -h\tau
\end{equation}
where the {\em regularized shape potenial} $W$ is
\begin{equation}\label{eq_W}
W(z) = \fr{|X(z)|}{\normtwo{z}^6}\left(m_1m_2\rho_{31}\rho_{23} + m_1m_3\rho_{12}\rho_{23} + m_2m_3\rho_{12}\rho_{31}\right) 
\end{equation}
Note that since $z$ is a homogeneous variable representing $[z]\in\S^5$, we have $z\ne 0$.  For a homogeneous coordinate representing a binary collision we will have exactly one of the variables $z_{ij}=0$ and $\normtwo{z}>0$.  Thus $\tilde H$ is nonsingular at these points and the binary collisions are regularized.  

\begin{theorem}\label{th_sphericalregularized}
The Hamiltonian flow of $\tilde H_{sph}$ on $T^*\R^+\times T^*\C^3_0$  has an invariant submanifold 
$T^*\R^+\times T^*_{sph,\cC}\C^3_0$ defined by $\re\metric{\eta,z} = 0$ and $z^2_{12}+z^2_{31}+z^2_{23}=0$.  The quotient of the restricted flow by  scaling   and  by translation of $\eta$ by $\bar z$  represents the zero total momentum three-body problem with regularized binary collisions, reduced by translations (but not by rotations).
\end{theorem}

The quotient space of $T^*_{sph,\cC}\C^3_0$ by these symmetries can be identified with $T^*\S(\cC)\simeq T^*\RP(3)$.  The regularizing map induces an 8-to-1 branched covering map $f_{sph}:\S(\cW)\into\S(\cW)$, that is, an 8-to-1 branched covering
$\RP^3\mapsto \S^3$.  The map is a diffeomorphism except where (exactly) one of the $z_{ij}=0$ and $X_{ij}=0$.   To describe the branching behavior, note that in the two-dimensional complex subspace $\cW$, the set where $X_{12}=0$ is a complex line which corresponds to a circle $S^1$ in the sphere $\S(\cW)$.  The preimage of this circle will be 2 circles in the projective space $\S(\cC)$.  Altogether, the map is branched over 3 circles, each circle having pre-image 2 circles   in the projective space
$\RP^3$.

\subsubsection{Quadratic Parametrization of $\cC$}\label{sec_quadraticCsph}
Instead of writing Hamilton's equations for $\tilde H_{sph}$ we will describe a parametrization of the cone $\cC$ which leads to a lower-dimensional system of equations.  There is nice 2-to-1 parametrization by quadratic polynomials which is related to the double covers of $\RP^3$ by $\S^3$, of $SO(3)$ by the unit quaterions, and of $SO(3)$ by $SU(2)$.

Define a 2-to-1 mapping $g:\C^2\into \cC \subset \C^3$ by
\begin{equation}\label{eq_quadraticmap}
g: \qquad z_{12} = 2i\, x_1 x_2\qquad z_{31}=x_1^2+x_2^2\qquad z_{23} = i(x_1^2-x_2^2)
\end{equation}
where $x_1,x_2\in \C$.    This can be seen as a variant of a map used by Waldvogel \cite{Waldvogelreg} in his regularization of the planar problem.  But here we are applying the idea to the homogeneous variables $X$ which makes it easier to blow-up triple collision later on.

By homogeneity there is an induced map $g_{sph}:\S^3\into S(\cC)$.  The induced map is given by the same formula except that    $x,z$ now denote  homogenous coordinates for the points of $\S^3,\S^5$.  (This  double covering map gives another way to see that $S(\cC)$ is diffeomorphic to the real projective space $\RP^3$.)  The map $g_{sph}$ can be motivated in several  ways. First, after omitting the factors of $i$, it resembles the formulas for parametrizing Pythagorean triples.  Next, write $x_1=u_1-i\,u_2, x_2 = u_3+ i\,u_4$  and define the unit quaternion
$u = u_1+i\,u_2+j\,u_3+k\,u_4$.  Then the familiar conjugation map $v\mapsto uv \bar u$ where $v$ is an imaginary quaternion defines a rotation
$R(x)$ on the 3-dimensional space of $v$'s.   Up to a permutation of the columns,  $R(x) = A(z)$, the matrix of (\ref{eq_Aofz}),
and hence the conjugation map defines a map $x \mapsto z$.  As a variation on this construction,   define the unitary  $x$-dependent matrix
$$U = \m{\bar x_1& x_2\\ -\bar x_2&x_1}  \in SU(2).$$
Then the adjoint representation $v\mapsto U(x)vU(x)^{-1}$ on $su(2)\simeq\R^3$ produces the same rotation $R(x)$. 

The composition $f\circ g_{sph}$ of the regularizing map and the quadratic paremetrization gives a 16-to-1 branched cover 
$\S^3\mapsto\S^3$, which becomes 8-to-1 over the binary collisions.  Each binary collision is represented by a circle in the range which  has 2 preimage circles for a total of 6 branching circles in the domain.  Using stereographic projection, it is possible to get some idea of the behavior of this remarkable, regularizing map.  Figure~\ref{fig_regmap3D} shows the projection of the three-sphere.  The three transparent surfaces are tori
representing the collinear configurations with a given ordering of the bodies along the line.  These intersect in 6 circles representing the binary collisions.  The figure shows thin tubes around each of these circles.

\begin{figure}[h]
\scalebox{1.0}{\includegraphics{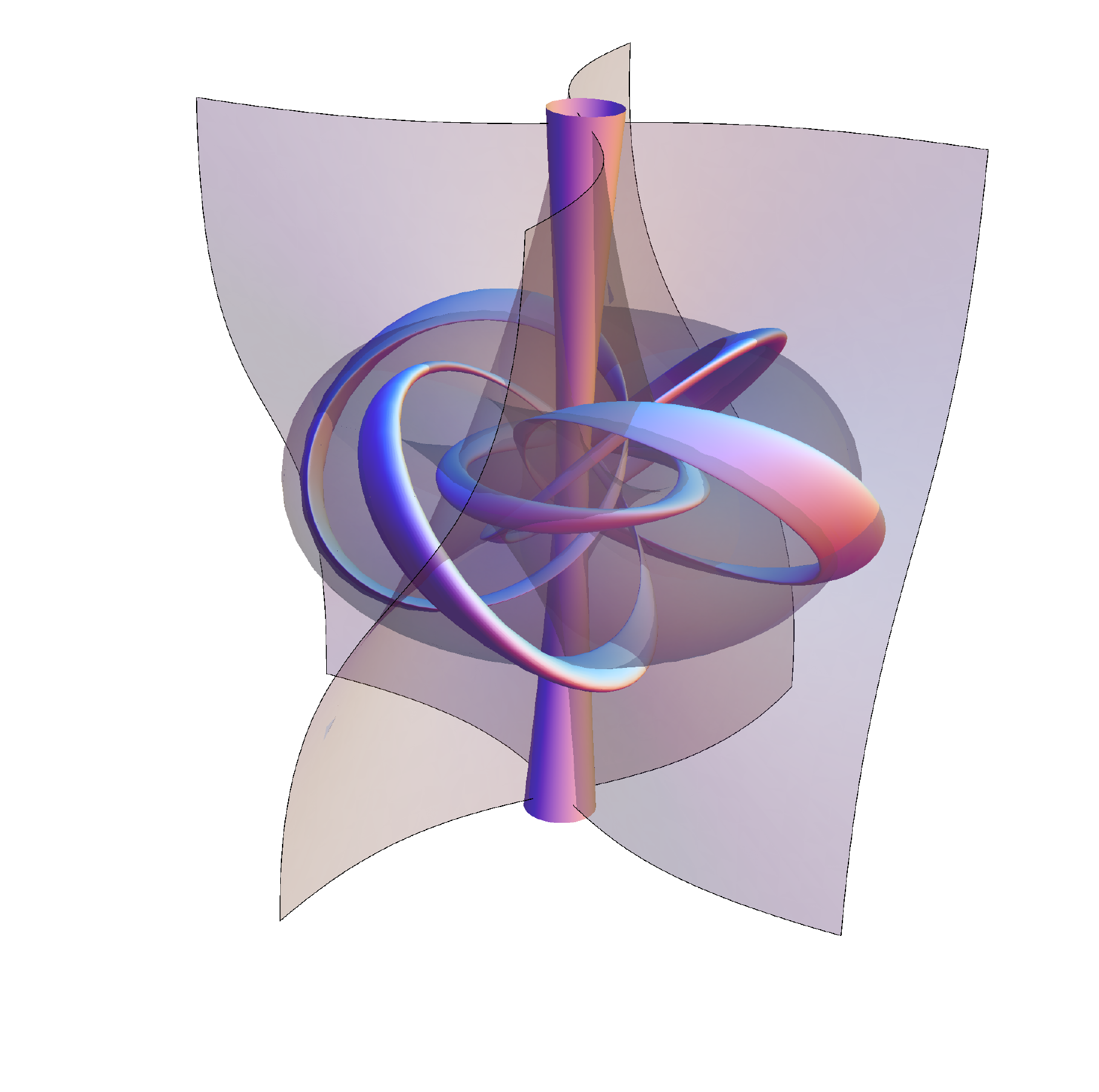}}
\caption{Stereographic projection of $\S^3$ showing the preimage under the regularizing map of the collinear configurations and small tubes around the binary collision circles.} \label{fig_regmap3D}
\end{figure}

To extend $g$ to a partially symplectic map $G:T^*\R^+\times T^*\C^2\into T^*\R^+\times \cC\times\C^{3*}$ we transform the momenta $\eta, y$ so that  $y = \eta  \overline{Df(z)}$or 
 $$\m{y_1& y_2} = \m{\eta_{12}& \eta_{31}& \eta_{23}}\m{-2i\bar{x}_2&-2i\bar{x}_1\cr 2\bar{x}_1&2\bar{x}_2\cr -2i\bar{x}_1&2i\bar{x}_2}$$
The value of $\eta$ is not uniquely determined but any two solutions will yield equivalent covectors and the same transformed Hamiltonian.  For example, we could take 
$$\eta_{12}=0\quad\eta_{31} = \fr14\left( \fr{y_1}{\bar{x}_1}+\fr{y_2}{\bar{x}_2}\right) \quad \eta_{23}=  \fr{i}4\left( \fr{y_1}{\bar{x}_1}-\fr{y_2}{\bar{x}_2}\right) .$$
$G$ restricts to $G:T^*\R^+\times T^*_{sph}\C^2\into T^*\R^+\times T^*_{sph,\cC}\C^{3*}$
where $T^*_{sph}\C^2 = \{(x,y):\re\metric{y,x}= 0\}$ and $T^*_{sph,\cC}\C^{3*}=\{(z,\eta):z\in\cC,\re\metric{\eta,z}=0\}$.

The regularized spherical Hamiltonian becomes
\begin{equation}\label{eq_Htildesphxy}
\begin{aligned}
\tilde H_{sph} &= 
\fr{\tau\, p_r^2}{2} + \fr{|X(x)|^2}{r^2\normtwo{x}^{12}}\left(\fr{|\pi'_1|^2\rho_{23}}{256m_1}+\fr{|\pi'_2|^2\rho_{31}}{256m_2}+  \fr{|\pi'_3|^2\rho_{12}}{256m_3}\right) - \fr{1}{r}W(x)  -h\tau\\
\pi'_1 &= y_1\bar{x_2}+ y_2\bar{x_1}\quad \pi'_2 =y_1\bar{x_2}-y_2\bar{x_1} \quad 
\pi'_3 = y_1\bar{x_1}- y_2\bar{x_2}\\
\rho_{12}&= |2x_1x_2|^2\quad\rho_{31}= |x_1^2+x_2^2|^2\quad\rho_{23}= |x_1^2-x_2^2|^2\\
\normtwo{z}^2 &=2\normtwo{x}^4= \rho_{12} +\rho_{31}+\rho_{23}\qquad |X(x)|^2 = \fr{m_1m_2\rho_{12}^2+m_1m_3\rho_{31}^2+m_2m_3\rho_{23}^2}{m_1+m_2+m_3}.
\end{aligned}
\end{equation}
Note that $\tilde H$ is invariant under the scaling symmetry $(x,y)\into(kx,k^{-1}y)$, $k>0$.   The corresponding Hamiltonian system on the $10$-dimensional space $T^*(\R^+\times \C^2)$ can be reduced to the expected $8$ dimensions by restricting to the invariant set $T^*\R^+\times T^*_{sph}\C^2$ and then passing to the quotient space under scaling.

\subsection{Projective Regularization}
Next we will get a regularized version of the reduced three-body problem.
Let $(r,X)$ be the radial-homogeneous  coordinates of section~\ref{sec_reduction}.   For a fixed angular momentum, we have the reduced Hamiltonian on $T^*\R^+\times T^*_{pr}\C^3$
$$
H_{\mu}(r,p_r,X,Z) = \fr12(p_r^2+\fr{\mu^2}{r^2}) +\fr{|X|^2}{r^2} K(Z)-  \fr1r V([X]).
$$
After making the Levi-Civita transformations, fixing an energy and changing timescale by the factor $\tau$ from (\ref{eq_tau}) we obtain a regularized reduced Hamiltonian
\begin{equation}\label{eq_Hreducedreg}
\tilde H_{\mu} = 
\fr{\tau \, p_r^2}{2}+\fr{
\tau\, \mu^2}{2r^2} + \fr{|X(z)|^2}{r^2\normtwo{z}^6}\left(\fr{|\pi_1|^2\rho_{23}}{8m_1}+\fr{|\pi_2|^2\rho_{31}}{8m_2}+  \fr{|\pi_3|^2\rho_{12}}{8m_3}\right) - \fr{1}{r}W(\xi)  -h\tau
\end{equation}
where the various quantities appearing in the formula are given by (\ref{eq_rhoij}), (\ref{eq_normXsq}), (\ref{eq_normzsq})and  (\ref{eq_W}).  The only difference between the spherical and projective Hamiltonians is the term involving $\mu^2$. We also impose  the extra constraint $\im\metric{\eta,z} = 0$ and there will be extra curvature terms in the differential equations.

To find the curvature terms we need to pull-back the Fubini-Study form under the regularizing map 
$X= f(z), X_{ij}=z_{ij}^2$.  The Fubini-Study metric on $z$-space is derived from the standard Hermitian metric on $\C^3$ by a formula analogous to (\ref{eq_FSmetric}).   We can express its restriction to $\cC$ in terms of a tangent vector field as we did in lemma~\ref{lemma_FSlemma}.  The analogous formula to (\ref{eq_FSwithT}) is
\begin{equation}\label{eq_FSzwithT}
\metrictwo{V,W}_{FS,\cC}  = \fr{\metrictwo{V,e}\metrictwo{e,W}}{\normtwo{z}^4}\qquad V,W\in T_X\cS
\end{equation}
where $e(z)$ is a Fubini-Study unit vectorfield tangent to $\cC$ and normal to $z$.  For example, observe that if $z\in \cC_0=\cC\setminus 0$ then the vectors $z,\bar z,  T$ form a Hermitian-orthogonal complex basis for $T_z\C^3$ where
\begin{equation}\label{eq_zcrosszbar}
T = {z}\times \bar z = (z_{31}{\bar z}_{23}-z_{23}{\bar z}_{31},z_{23}{\bar z}_{12}-z_{12}{\bar z}_{23},z_{12}{\bar z}_{31}-z_{31}{\bar z}_{12}).
\end{equation}
Hence we can take  $e = \fr{\normtwo{z}}{\normtwo{T}}T =  ({z}\times \bar z)/\normtwo{z}$.  This gives
\begin{equation}\label{eq_FStensorz}
\metrictwo{.,.}_{FS,\cC} =   \fr{\conj\Sigma \tensor {\Sigma}}{\normtwo{z}^4}  \end{equation}
where $\Sigma$ is given by 
any of the following formulas
\begin{equation}\label{eq_Sigma}
\begin{aligned}
\Sigma &=\fr{ \metrictwo{z\times\bar z,dz}}{\normtwo{z}}
= \frac{\normtwo{z}(z_{12}dz_{31}-z_{31}dz_{12})}{z_{23}} \\
&=  \frac{\normtwo{z}(z_{23}dz_{12}-z_{12}dz_{23})}{z_{31}}= \frac{\normtwo{z}(z_{31}dz_{23}-z_{23}dz_{31})}{z_{12}}.
\end{aligned}
\end{equation}
For example, the first version is just $\Sigma = \metrictwo{e,dz}$ and  the second is  obtained by eliminating  $z_{23}, dz_{23}$ using the equations $z^2_{23} = -z^2_{12}-z^2_{31}$ and $z_{23}dz_{23} = -z_{12}dz_{12}-z_{31}dz_{31}.$

Using these formulas, we find that the pull-back of the Fubini-Study metric on $\cW$ is a conformal multiple of the Fubini-Study metric on $\cC$\begin{lemma}\label{lemma_FSconformal}
The pull-back of the Fubini-Study metric on $\cW$ is given by  
$$f^*\metric{.,.}_{FS,\cW} = \lambda(z)\metrictwo{.,.}_{FS,\cC}$$
 where the conformal factor is
\begin{equation}\label{eq_lambda}
\lambda =  \fr{4 m_1m_2m_3\,\rho_{12} \rho_{31} \rho_{23}\normtwo{z}^2}{m|X(z)|^4} = \fr{4 m\,m_1m_2m_3 (\rho_{12}+ \rho_{31}+\rho_{23})\,\rho_{12} \rho_{31} \rho_{23}}{(m_1m_2 \rho_{12}^2+m_1m_3 \rho_{31}^2+m_2m_3 \rho_{23}^2)^2}
\end{equation}
and where  $\rho_{ij}=|z_{ij}|^2$.
\end{lemma}
\begin{proof}
Equation (\ref{eq_FStensorX}) shows that we need to compute the pullback $f^*\sigma$, where $\sigma$ is given by (\ref{eq_sigma}).
Using the first formula for $\sigma$ gives
$$f^*\sigma = 2z_{12}^2z_{31}dz_{31} - 2z_{31}^2z_{12}dz_{31} =  2z_{12}z_{31}z_{23}\Sigma.$$
Hence
$$f^*\metrictwo{.,.}_{FS,\cW} = \fr{m_1m_2m_3}{m |X(z)|^4}\,f^* \conj{\sigma}\tensor{f^*\sigma}
= \fr{4 m_1m_2m_3}{m |X(z)|^4}|z_{12}|^2|z_{31}|^2|z_{23}|^2 \conj{\Sigma}\tensor{\Sigma}.$$
Now use (\ref{eq_normXsq}), (\ref{eq_normzsq}) and (\ref{eq_FStensorz}) to get the formula in the proposition.
\end{proof}

Similarly we can pull-back the Fubini-Study cometric on $\cW$ and compare it with the dual Fubini-Study metric on $\cC$.  The formula analogous to (\ref{eq_dualFSwithT}) is
\begin{equation}\label{eq_FSC}
\normtwo{\eta}^2_{FS,\cC} = |\pairing{\eta,e}|^2= \fr{|\pairing{\eta,z\times\bar z}|^2}{\normtwo{z}^2}\qquad \eta\in T^*_{z,pr}\C^3.
\end{equation}
This is a degenerate quadratic form, invariant under $z$-translation of $\eta$, which represents the   Fubini-Study cometric on $\cC$.  

The next lemma relates this to the pull-back of the Fubini-Study cometric on $\cW$ and hence, to the shape kinetic energy.
\begin{lemma}\label{lemma_FSzdual}
The pull-back of the Fubini-Study cometric on $\cW$ is 
$$F^*\normtwo{\,.\,}^2_{FS,\cW} = \lambda^{-1}\normtwo{\,.\,}^2_{FS,\cC}$$
where $\lambda$ is given by (\ref{eq_lambda}).  Hence the shape kinetic energy in regularized coordinates is 
$$\fr12 \lambda^{-1}\normtwo{\eta}^2_{FS,\cC} = \fr12\fr{|\pairing{\eta,z\times\bar z}|^2}{\lambda\,\normtwo{z}^2}.$$
\end{lemma}
\begin{proof}
Equation (\ref{eq_FSdualalpha}) shows that we need to compute the pullback $F^*\alpha$, where $\alpha$ is given by (\ref{eq_alpha}).
Using the second formula for $\alpha$ gives
$$\fr{|z_{23}|^2}{|X|^2}F^*\alpha = \fr{(\eta_{31}\bar z_{12}-\eta_{12}\bar z_{31})z_{23}}{2\bar z_{12}\bar z_{31}\bar z_{23}}$$
and there are two similar equations from the third and fourth formulas.  Adding these gives
$$F^*\alpha = \fr{|X(z)|^2}{\normtwo{z}^2}\pairing{\bar\eta,z\times\bar z}.$$
Therefore,
$$F^*\normtwo{\eta}^2_{FS,\cW} = \fr{m |X(z)|^4 |\pairing{\eta,z\times\bar z}|^2}{4m_1m_2m_3\rho_{12} \rho_{31} \rho_{23}\normtwo{z}^4}=
\fr{m |X(z)|^4}{4m_1m_2m_3\rho_{12} \rho_{31} \rho_{23}\normtwo{z}^2}\normtwo{\eta}^2_{FS,\cC}.$$
Comparing with the formula for $\lambda$ completes the proof.
\end{proof}

It follows from the lemma that we have an equivalent reduced, regularized Hamiltonian
\begin{equation}\label{eq_HreducedregFS}
\begin{aligned}
\tilde H_{\mu} &= 
\fr{\tau \, p_r^2}{2}+\fr{\tau\, \mu^2}{2r^2} + \fr{\tau\normtwo{\eta}_{FS,\cC}^2}{2\lambda(z) r^2} - \fr{1}{r}W(\xi)  -h\tau\\
&=\fr{\tau \, p_r^2}{2}+\fr{\tau\, \mu^2}{2r^2} + \fr{m |X(z)|^4 |\pairing{\eta,z\times\bar z}|^2}{8m_1m_2m_3 r^2\normtwo{z}^{10}} - \fr{1}{r}W(\xi)  -h\tau
 \end{aligned}
\end{equation}

The factor of $\l$ in the Fubini-Study two-form and the factor of $\l^{-1}$ in the shape kinetic energy cancel out in the interior product defining the curvature term.  Remembering the timescale factor $\tau$ we find that the the curvature  term is
 \begin{equation}
 \label{eq_curvterm1}
 T_{curv} = -\fr{2\mu\tau}{r^2} \,i\,\eta
 \end{equation} 
which is added to the right hand side (i.e. to $-\fr{\partial H}{\partial z}$) of the Hamilton's
equation for $\dot \eta$.

\begin{theorem}\label{th_reducedregularized}
The Hamiltonian flow of $\tilde H_{\mu}$ on $T^*\R^+\times T^*\C^3_0$  has an invariant set $T^*\R^+\times T^*_{pr,\cC}\C^3$ where $\metric{\eta,z} = 0$ and $z^2_{12}+z^2_{31}+z^2_{23}=0$ with symplectic structure given by the restriction of the standard form minus $2\mu\lambda\Omega_{FS}$.  The quotient of the restricted flow by the complex scaling symmetry and by $\bar z$-translations of $\eta$ represents the three-body problem with zero total momentum and angular momentum $\mu$, with regularized binary collisions, reduced by translations and rotations.
\end{theorem}

The regularized, reduced Hamiltonian $\tilde H_{\mu}$, together with the curvature term gives a system of differential equations on the $14$-dimensional space $T^*(\R^+\times \C^3)$ with variables $(r,p_r,z,\eta)$.   
The six-dimensional quotient space of $T^*\R^+\times T^*_{pr,\cC}\C^3$ is diffeomorphic to $T^*\R^+\times T^*\bP(\cC)$.
Instead of writing these $14$-dimensional differential equations we will describe several ways to parametrize the regularized shape sphere $P(\cC)$ to arrive at lower-dimensional systems of equations.

\subsubsection{Quadratic Parametrization of the Regularized Shape Sphere}\label{sec_quadraticCpr}
We can parametrize $\cC$ using the same quadratic map  $g:\C^2\into \cC \subset \C^3$ 
\begin{equation*}
z_{12} = 2i\, x_1 x_2\qquad z_{31}=x_1^2+x_2^2\qquad z_{23} = i(x_1^2-x_2^2)
\end{equation*}
as in section~\ref{sec_quadraticCsph}.  Since $g$ is homogeneous with respect to complex scaling, it induces a map $g_{pr}: \CP^1\into P(\cC)$ from the projective line onto $P(\cC)$.   Although $g$ and the induced map $g_{sph}$ of $\S^3$ in section~\ref{sec_quadraticCsph} are both 2-to-1, the extra quotienting makes $g_{pr}$  a diffeomorphism.  This shows again that $P(\cC)$ is diffeomorphic to the two-sphere.  The same partially symplectic extension $G:T^*\R^+\times T^*\C^2\into T^*\R^+\times \cC\times\C^{3*}$ restricts to a map  $G:T^*\R^+\times T^*_{pr}\C^2\into T^*\R^+\times T^*_{pr,\cC}\C^{3}$ where $T^*_{pr}\C^2 = \{(x,y):\metric{y,x}= 0\}$ and $T^*_{pr,\cC}\C^{3}=\{(z,\eta):z\in\cC,\metric{\eta,z}=0\}$.

If we use (\ref{eq_HreducedregFS}) together with the formula (\ref{eq_FSC}) for the dual Fubini-Study metric we obtain, after some simplification, 
the reduced, regularized Hamiltonian 
\begin{equation}\label{eq_Htildemuxy}
\begin{aligned}
\tilde H_{\mu} &= 
\fr{\tau\, p_r^2}{2} + \fr{\tau\, \mu^2}{2r^2} +  \fr{\tau}{4\lambda r^2}|y_1x_2-x_1y_2|^2 - \fr{1}{r}W(x)  - h\tau \\
W(x) &= \fr{|X(x)|}{\normtwo{x}^{12}}\left(m_1m_2\rho_{31}\rho_{23} + m_1m_3\rho_{12}\rho_{23} + m_2m_3\rho_{12}\rho_{31}\right) \\
\rho_{12}&= |2x_1x_2|^2\quad\rho_{31}= |x_1^2+x_2^2|^2\quad\rho_{23}= |x_1^2-x_2^2|^2\\
2\normtwo{x}^4 &=\rho_{12} +\rho_{31}+\rho_{23}\qquad |X(x)|^2 = \fr{m_1m_2\rho_{12}^2+m_1m_3\rho_{31}^2+m_2m_3\rho_{23}^2}{m_1+m_2+m_3}
\end{aligned}
\end{equation}
We have
\begin{equation}
\label{eq_reparameterization}
\tau = \fr{\rho_{12}\rho_{31}\rho_{23}}{8\normtwo{x}^{12}}\qquad \fr{\tau}{\lambda} =  \fr{m|X(x)|^4}{64 m_1m_2m_3\,\normtwo{x}^{16}}.
\end{equation}

We have the complex constraint $\pairing{y,x} = 0$ and the system is invariant under complex scaling symmetry $(x,y)\into(kx,y/\bar k)$, $k\in\C_0$.  Applying the constraint and passing to the quotient space reduces the dimension from 10 to 6.  As usual, Hamilton's differential equations will have a curvature term
 $$T_{curv} =  -\fr{2\mu\tau}{r^2} \,i\,y$$
 added to the $\dot y$ equation.
 
 \subsubsection{Dynamics in regularized affine coordinates. }
 \label{sec_regularizedaffine}
As in section~\ref{sec_parametrizingPW} we can use affine local coordinates on $\CP^1$.    Every projective  point $[x_1,x_2] \in\CP^1$ with $x_1\ne 0$ has a representative of the form
$$[x_1,x_2] = [1, z]= [1,x+i\,y]$$
where $x,y\in\R$.
The appropriate momentum substitution is
$$y_1=-\bar{z}\zeta \qquad y_2 = \zeta.$$
where $\zeta =  \alpha+i\beta\in \C^*$ is a momentum vector dual to $z$.

We get a Hamiltonian system with 6 degrees of freedom
\begin{equation}
\begin{aligned}
\label{eq_affineHamiltonian}
\tilde H_{\mu} &= 
\fr{\tau\, p_r^2}{2} + \fr{\tau\, \mu^2}{2r^2} +  \fr{\tau}{2\lambda r^2}(1+x^2+y^2)^2(\a^2+\b^2)- \fr{1}{r}W(x,y) - h\tau \\
W(x,y) &= \fr{|X(x,y)|}{(1+x^2+y^2)^{6}}\left(m_1m_2\rho_{31}\rho_{23} + m_1m_3\rho_{12}\rho_{23} + m_2m_3\rho_{12}\rho_{31}\right) \\
\rho_{12}&= 4(x^2+y^2)\quad \rho_{31}=(1+x^2-y^2)^2+4x^2y^2\quad\rho_{23}=  (1-x^2+y^2)^2+4x^2y^2\\
\normtwo{x}^2 &= 1+x^2+y^2 \qquad |X(x,y)|^2 = \fr{m_1m_2\rho_{12}^2+m_1m_3\rho_{31}^2+m_2m_3\rho_{23}^2}{m_1+m_2+m_3}
\end{aligned}
\end{equation}
 The Fubini-Study form becomes
 $$\Omega_{FS} = \fr{\,dx \wedge dy}{ (1 + x^2 + y^2)^2}$$
 and
$$
\tau = \fr{\rho_{12}\rho_{31}\rho_{23}}{8(1+x^2+y^2)^{6}}\qquad \fr{\tau}{\lambda} =  \fr{m|X(x,y)|^4}{64 m_1m_2m_3\,(1+x^2+y^2)^{8}}.
$$

Hamilton's equations with the curvature term are 
\begin{equation}
\begin{aligned}
\dot r &= \tau p_r  \\
\dot p_r &= \fr{1}{r^3} [ \fr{\tau}{\lambda}(1+|z|^2)^2|\zeta|^2 +  \tau \mu^2 ]- \fr1{r^2}W(x,y) \\
\dot x   &=  \fr{\tau}{ \lambda} \fr{(1+x^2+y^2)^2}{r^2} \alpha  \\
\dot y   &= \fr{\tau}{\lambda} \fr{(1+x^2+y^2)^2}{r^2}\beta  \\
\dot \alpha  &= \fr{1}{r^2} [-\Lambda_x+ 2 \tau \mu \beta  -\fr{\tau_x } {2} \mu^2]
 + \fr{1}{r}W_x  - \tau_x [ \fr{p_r^2}{2} - h ]  \\
 \dot \beta   &= \fr{1}{r^2} [-\Lambda_y - 2 \tau \mu \alpha  -\fr{\tau_y } {2} \mu^2]
 + \fr{1}{r}W_y  - \tau_y [ \fr{p_r^2}{2} - h ] 
\end{aligned}
\end{equation}
where
\begin{equation*} 
\Lambda(x,y,\a,\b) = \fr{\tau}{2\lambda}(1+x^2+y^2)^2 (\a^2+\b^2). 
\end{equation*}

 \subsubsection{Dynamics in regularized spherical coordinates.}
 \label{sec_regularizedshapesphereround}
Instead of using projective or local affine coordinates, one can map the regularized shape sphere to the unit sphere in $\R^3$. 
A particularly elegant way to do this is to use the diffeomorphism between $\cC$ and $SO(3)$ described in section~\ref{sec_geometryC}.

Given $z\in\cC$ we write $z = a+i\, b$ where $a,b\in\R^3$ and then define $c = a\times b\in\R^3$.  We saw that the matrix
 \begin{equation*}
 A(z) = \fr1s\m{a_{12}&a_{12}&c_{12}/s\\a_{31}&b_{31}&c_{31}/s\\a_{23}&b_{23}&c_{23}/s}
 \end{equation*}
is in $SO(3)$, where $s^2=|z|^2/2 = |a|^2 = |b|^2 = |c|$

We will work homogeneously and define a map $g:\cC\into\R^3$
$$g(z) = c = \re(z)\times \im(z).$$
By homogeneity, there is an induced map $g_{pr}:\bP(\cC)\into \S(\R^3)\simeq \S^2$ where we view $z$ and $c$ as homogeneous coordinates with respect to complex and positive real scaling respectively.  

The orthogonality of the matrix $A(z)$ can be used to derive some useful formulas.  
Since the rows as well as the columns are unit vectors we find
$$\rho_{ij} = |z_{ij}|^2 = a_{ij}^2+b_{ij}^2 =  \fr{|c|^2 - c_{ij}^2}{|c|}$$
which gives the beautiful formulas
\begin{equation}\label{eq_rhoijc}
\rho_{12} = \fr{c_{31}^2+c_{23}^2}{|c|}\qquad \rho_{31} = \fr{c_{12}^2+c_{23}^2}{|c|}\qquad \rho_{23} = \fr{c_{12}^2+c_{31}^2}{|c|}
\end{equation}
for the homogeneous mutual distances.  Similar formulas were given by Lemaitre \cite{Lemaitre2}.

Next, consider the quantity
$$\bar{z}_{12}z_{31} = a_{12}a_{31}+b_{12}b_{31} + i\,(a_{12}b_{31}-a_{31}b_{12})
= (a_{12},b_{12})\cdot (a_{31},b_{31}) + i\,c_{23}.$$
Using the orthogonality of the rows we can express this entirely in terms of $c$.  We find
$$\bar{z}_{12}z_{31} = -\fr{c_{12}c_{31}}{|c|} + i\,c_{23}\quad \bar{z}_{23}z_{12} = -\fr{c_{23}c_{12}}{|c|} + i\,c_{31}\quad \bar{z}_{31}z_{23} = -\fr{c_{31}c_{23}}{|c|} + i\,c_{12}.$$
These last formulas allow us to write down local inverses for $g_{pr}$.  Namely consider the map
$h_{12}:\R^3\into\C^3$
$$\begin{aligned}
h_{12}(c) &=  |c| \bar{z}_{12}(z_{12},z_{31},z_{23}) = |c| (\bar{z}_{12}z_{12},\bar{z}_{12}z_{31},\bar{z}_{12}z_{23})\\
&= (c_{31}^2+c_{23}^2, -c_{12}c_{31} + i\,|c|c_{23},-c_{12}c_{23} - i\,|c|c_{31}).
\end{aligned}$$
If $z_{12}\ne 0$, then $h_{12}(c)$ represents the same projective point in $\bP(\cC)$ as $z$ does so $h_{12}(c)$ give a local inverse for the projective map $g_{pr}$.  There are similar partial inverses $h_{31}, h_{23}$.

To find the regularized, reduced Hamiltonian system we need to convert the Fubini-Study metric and its dual norm
(i.e. cometric)  to $c$-coordinates.
The spherical analogue of the Fubini-Study metric is the spherical metric
$$\metric{.,.}_{sph} = \fr{|c|^2\metric{dc,dc}-\metric{dc,c}\metric{c,dc}}{|c|^4}= \fr{|c\times dc|^2}{|c|^4}$$ 
where we are using the Euclidean inner product on $\R^3$.
We will see that   
$$g^*\metric{.,.}_{sph} = 2\metrictwo{.,.}_{FS,\cC} = \fr{2|\metrictwo{z\times\bar z,dz}|^2}{\normtwo{z}^6}.$$
To see this, note that $z\times \bar z = -2i\, a\times b = -2i\,c$.  Hence
$$dc = \fr{i}{2}(dz\times \bar z + z\times d\bar z).$$
This, together with the fact that $\metrictwo{z,\bar z}=0$ on $\cC$ leads, after some algebra, to the pull-back formula.  
Correspondingly,  the Euclidean solid angle form pulls back to twice the   Fubini-Study form, hence
 $$\lambda\Omega_{FS,\C} = g^*\fr{\lambda}{2|c|^3}(c_1 dc_2\wedge dc_3 +c_2 dc_3\wedge dc_1 +c_3 dc_1\wedge dc_2).$$

Let $\g\in\R^{3*}$ be a dual momentum vector to $c\in\R^3$.  From the spherical scaling we will have
$\g\cdot c = 0$.  If we split the momentum vector $\eta$
into real and imaginary parts, $\eta = u+i\, v$ then the momenta transform via
$$u=b\times \gamma\quad v=-a\times\gamma\qquad\qquad \gamma = -\fr{u\cdot c}{|c|^2}\,a-\fr{v\cdot c}{|c|^2}\,b.$$
From this we find that the dual spherical norm
$\norm{\g}^2_{sph} = |\g \times c|^2 = |c|^2|\g|^2$ 
corresponds to $\fr12\normtwo{.}^2_{FS,\cC}$.  So we get the reduced, regularized Hamiltonian

\goodbreak
\begin{equation}\label{eq_Htildemuc}
\begin{aligned}
\tilde H_{\mu} &= 
\fr{\tau \, p_r^2}{2}+\fr{\tau\, \mu^2}{2r^2} + \fr{\tau|c|^2|\g|^2}{\lambda(c) r^2} - \fr{1}{r}W(c)  -h\tau \\
W(c) &= \fr{|X(c)|}{8|c|^6}\left(m_1m_2\rho_{31}\rho_{23} + m_1m_3\rho_{12}\rho_{23} + m_2m_3\rho_{12}\rho_{31}\right) \\
\rho_{12}&= c_{31}^2+c_{23}^2\quad\rho_{31}= c_{12}^2+c_{23}^2\quad\rho_{23}= c_{12}^2+c_{31}^2\\
2|c|^2 &=\rho_{12} +\rho_{31}+\rho_{23}\qquad |X(c)|^2 = \fr{m_1m_2\rho_{12}^{2}+m_1m_3\rho_{31}^{2}+m_2m_3\rho_{23}^{2}}{m_1+m_2+m_3}
\end{aligned}
\end{equation}
We have
$$\tau = \fr{\rho_{12}\rho_{31}\rho_{23}}{8|c|^6}\qquad \fr{\tau}{\lambda} =  \fr{m|X(c)|^4}{64 m_1m_2m_3\,|c|^8}.$$
Here we have redefined $\rho_{ij}$ to eliminate the factors of $|c|$ and placed these factors elsewhere in the formulas.  The curvature term is
\begin{equation}
\label{eq_curvatureterm_spherical}
T_{curv} = \fr{2 \mu\tau}{|c| r^2} \, \gamma \times c.
 \end{equation}

\subsection{Visualizing the Regularized Shape Sphere - LeMaitre's Conformal Map}
The map of projective curves $f_{pr}:\bP(\cC)\into \bP(\cW)$, induced by the squaring map,  can be visualized as a map of the two-sphere into itself.  Indeed this is the point of view taken by Lemaitre in \cite{Lemaitre2}.

The map is a four-to-one branched covering map with octahedral symmetry (see figure~\ref{fig_lemaitre}).  The map is generically four-to-one except at the binary collision points where it is two-to-one.  In the figure, each octant of the regularized sphere maps to one or the other hemisphere of the unregularized sphere.  Thus, for example, the north pole of the unregularized sphere (representing a Lagrangian, equilateral central configuration) has four preimages which lie in alternate octants.  Each binary collision point on the equator of the unregularized shape sphere, has two preimages, which lie on a coordinate axes of the regularized sphere.

\begin{figure}[h]
\scalebox{0.7}{\includegraphics{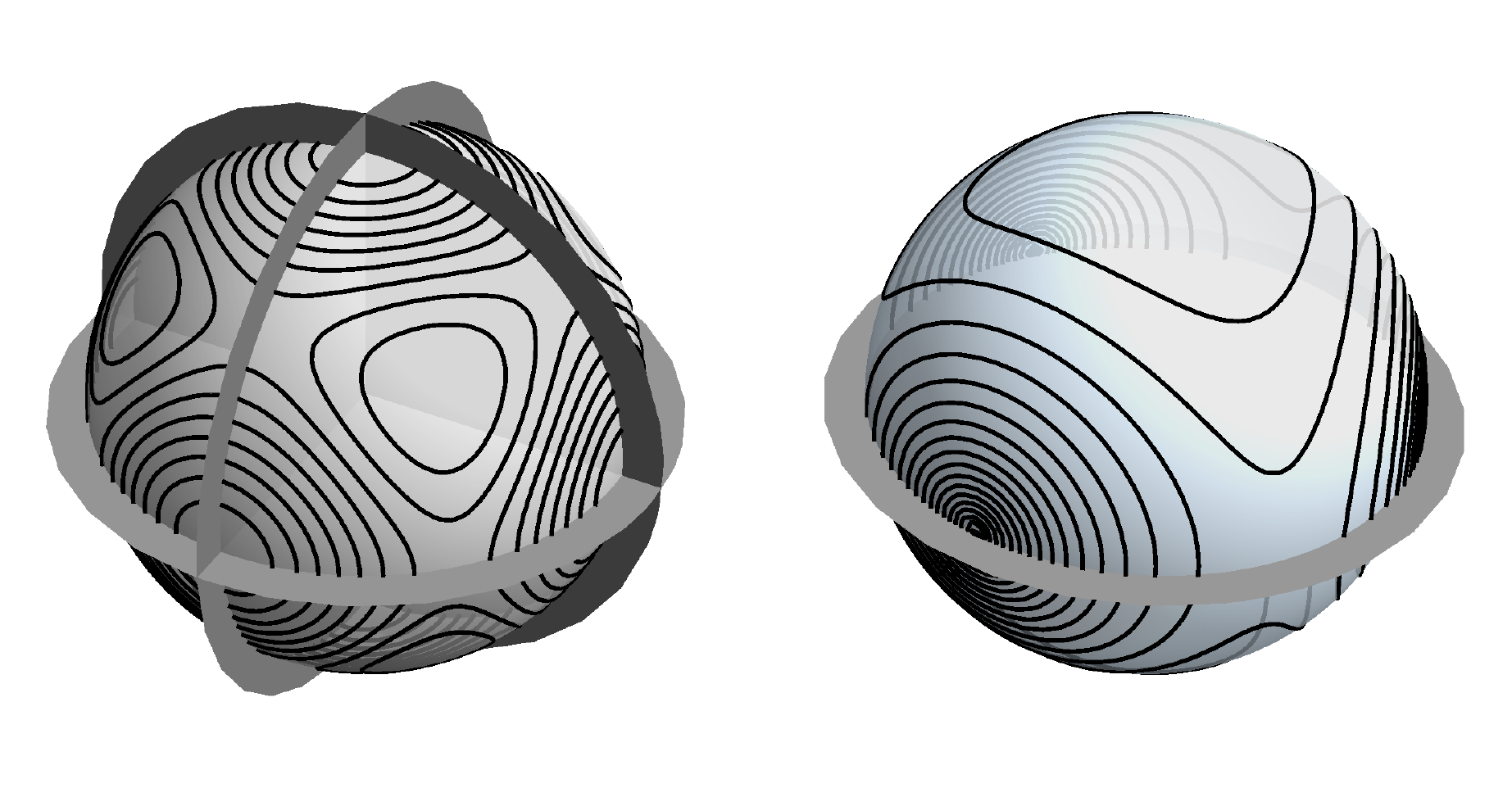}}
\caption{The regularizing map is a four-to-one branched cover of the two-sphere with octahedral symmetry.  Each octant of the regularized shape sphere (left) maps onto a hemisphere of the unregularized shape sphere (right).  The planes represent collinear configurations.  The figure also shows level curves of the unregularized shape potential and their preimages in the equal mass case.}  \label{fig_lemaitre}
\end{figure}

The three-dimensional sphere of figure~\ref{fig_regmap3D} is just the preimage of the regularized two-sphere sphere in figure~\ref{fig_lemaitre} under a Hopf-map.  Each point of the two-sphere determines a circle in the three-sphere.   The three large tori in figure~\ref{fig_regmap3D} are the preimages of the collinear circles in the two-sphere (where the coordinate planes cut the sphere).  The six tubes in figure~\ref{fig_regmap3D} are the preimages of small circles around the binary collision points (where the coordinate axes cut the sphere).

\section{Blowing Up Triple Collision} 
Our systematic use of the radial coordinate $r$ together with the  homogeneous coordinates used to describe the shape make it 
easy to implement McGehee's method for blowing-up total collision.  
We need only rescale    momenta    and change the timescale.  The changes 
can be made before or after reduction.   The changes  are non-canonical so 
destroy the Hamiltonian character of the equations.  We will describe the general method for the rotation-reduced and unreduced cases and then make some comments on the results of applying it to some of the Hamiltonians described above.

\subsection{Before Reduction} \label{subsec_beforereduction}
Consider a Hamiltonian of the general form
\begin{equation}\label{eq_generalH}
H(r,p_r,X,Y) =  \fr1{2r^2}B(X)(Y,Y) - \fr1r V(X) + [ \fr12 A(X)p_r^2 - C(X)]
\end{equation}
when expanded in powers of $r$.  This covers the rotation-unreduced Hamiltonian $H_{sph}$  of section~ \ref{sec_spherical} and the corresponding regularized Hamiltonians $\tilde H_{sph}(r,p_r,z,\eta)$ and $\tilde H_{sph}(r,p_r,x,y)$ of section~\ref{sec_sphericalregularization} (after changing the names of the variables).
For the unregularized Hamiltonian $H_{sph}$ we have 
$$A(X)=1\qquad C(X) = 0$$
 while for the regularized Hamiltonians $\tilde H_{sph}$ we have
$$A(X) = \tau(X)\qquad C(X) = h\,\tau(X).$$
The quantity $B(X)(Y,Y)$ represents the non-radial part of the kinetic energy.  It is a quadratic form in $Y$ 
which we  represent by a symmetric matrix
$B(X)$ depending on $X$. The dependence of $B$ on $X$ must  also   
be quadratic since $H$ must be 
 homogeneous of degree $0$ with respect to
the scaling $(X,Y) \mapsto (k X, \fr{1}{k} Y)$. 

Let $f(r)$ be a positive, real-valued function.  We will introduce a new timescale such that ${~}' = f(r)\dot{~}$.  The usual choice is McGehee's scaling factor $f_1(r) = r^{\fr32}$ but we will also consider $f_2(r) = \left(\fr{r}{r+1}\right)^{\fr32}$ which has better behavior for large $r$.  
(With the 1st choice solutions can reach  $r = \infty$ in finite time.)
For any such $f(r)$ we replace $(p_r,Y)$ by rescaled momentum variables 
\begin{equation}\label{eq_momentumscaling}
v = \fr{f(r)p_r}{r}\qquad \a = \fr{f(r)Y}{r^2}.
\end{equation}
The shape variable $X$ remains the same. 
When we make these substitutions of independent and dependent variables
 in the  Hamilton's differential equations resulting from  (\ref{eq_generalH}) we get
\begin{equation}\label{eq_blowup}
\begin{aligned}
r' &= A(X)vr\\
v'&=\fr12(1+r(\ln \nu)_r)A(X) v^2+B(X)(\a,\a) - \nu(r) V(X)\\
X'&= B(X)\a\\
\a'&=-\fr12(1-r(\ln \nu)_r)A(X)v\a-\fr12A_X\,v^2 - \fr12B_X(\a,\a) +\nu(r)V_X+r\nu(r)C_X
\end{aligned}
\end{equation}
where $\nu(r) = f(r)^2/r^3$ and where the subscripts denote differentiation.
 For McGehee's scaling $f(r) =f_1(r) =  r^{\fr32}$ we have $\nu(r) =1, (\ln \nu)_r = 0$ and the equations simplify considerably.  For $f_2(r)$ we have 
$\nu(r) = (1+r)^{-3}$  and both $\nu$ and $(\ln \nu)_r$ are still smooth all the way down to $r=0$.

Writing the energy equations $H_{sph}=h$ or $\tilde H_{sph} = 0$  in terms of the rescaled momenta gives
\begin{equation}\label{eq_blownupenergy}
\fr12 A(X) v^2+\fr12B(X)(\a,\a) -\nu(r)V(X) = r\nu(r) C(X).
\end{equation}

For example if use the $r^\fr32$ rescaling with $H_{sph}$, we have
$$A = 1\qquad B(X) = |X|^2B_0\qquad C = 0\qquad V(X) = |X|\sum_{i<j}\fr{m_i,m_j}{|X_{ij}|}.$$
where $B_0$ is the constant symmetric matrix (\ref{eq_Bmatrix}).
 We get the blown-up differential equations
$$\begin{aligned}
r' &= vr\\
v'&=\fr12 v^2-|X|^2B_0(\a,\a) + V(X)\\
X'&= |X|^2 B_0 \a\\
\a'&=-\fr12 v\a - B_0(\a,\a)X + V_X
\end{aligned}
$$
with the energy relation
$$\fr12 v^2+\fr12B_0(X)(\a,\a) -V(X) = rh.$$

The regularized equations arising from $\tilde H_{sph}$ are considerably more complicated due to the $B(X)$ terms (or rather the $B(z)$ or $B(x)$ terms.)  Instead of writing them explicitly we will just make some observations about them.  Consider, for example, $\tilde H_{sph}(r,p_r,x,y)$ from (\ref{eq_Htildesphxy}).  $B(x)$ will be a complicated, $4\times 4$ real matrix arising from the second term in (\ref{eq_Htildesphxy}).  The phase space before blow-up is 
$T^*\R^+\times T^*\C^2\simeq (0,\infty)\times\R\times\C^2\times\C^2$. In addition to the energy relation $\tilde H_{sph} = 0$, we have $\re\pairing{y,x} = 0$ and the scaling symmetry by positive real numbers so there is an induced flow on an quotient manifold of real dimension $7$.  After blow-up we have variables $(r,v,x,\a)\in [0,\infty)\times\R\times\C^2\times\C^2$, where we have extended the flow to the {\em collision manifold} where $r=0$, which is an invariant set for the differential equations.   We have a real-analytic vectorfield on this manifold-with-boundary.  Imposing the constraints and passing to the quotient under scaling gives a real-analytic vectorfield on an $7$-dimensional manifold-with-boundary representing the planar three-body problem on a fixed energy manifold, with all binary collisions regularized and with triple collision blown-up.  Note, in particular that the regularization of binary collisions passes smoothly to the boundary.

We claim that if the timescale factor $f(r) = f_2(r) = (r/(r+1))^{\fr32}$ is used, then the differential equations define a complete flow on  $[0,\infty)\times\R\times\C^2\times\C^2$ and hence the induced $7$-dimensional flow is complete.  Since the differential equations are smooth, the only obstruction to completeness would be orbits which become unbounded in finite time.  It is  well-known that, with the usual timescale, such orbits do not exist for the three-body problem.  It follows that if we use only bounded time-rescaling factors, the same will hold for the modified differential equations.  McGehee's factor $r^\fr32$ is unbounded and it is possible for orbits to escape in finite time.  Indeed, there are solutions of the three body problem for which $r(t) = O(t)$ as $t\into \infty$ with respect to the usual time-scale and these will reach infinity in finite rescaled time.  The factor $f_2$, while producing less elegant differential equations, eliminates this problem.

\subsection{After Reduction} \label{subsec_afterreduction}
The rotation-reduced Hamiltonians $H_\mu$ and their many regularized forms $\tilde H_\mu$ have the general form
\begin{equation}\label{eq_generalHmu}
H_{\mu} (r,p_r,X,Z) =  \fr1{2r^2} [B(X)(Z, Z) + A(X) \mu^2]    
 - \fr1r V(X) +  [\fr12 A(X) p_r^2 -  C(X)]
\end{equation}
(after changing the names of the variables).
The only new term  here,  when compared to the Hamiltonian of  section~\ref{subsec_beforereduction},
 is the  quadratic  term in the angular momentum $\mu$.
We have a momentum constraint $\pairing{Z,X} = 0$ and there will be a curvature term
$$T_{curv} = - \fr{2 \mu  b (X) }{r^2}  i Z$$
added to the $\dot Z$ equation.
As in section~\ref{subsec_beforereduction} the unregularized Hamiltonians $H_{\mu}$ we have 
$$A(X)=1\qquad C(X) = 0$$
 while for the regularized Hamiltonians $\tilde H_{\mu}$ we have
$$A(X) = \tau(X)\qquad C(X) = h\,\tau(X).$$
As in the last section, the variables $X, Z$  can denote either homogeneous coordinates
on the cotangent bundle of projective space,  before or after
Levi-Civita transformation, or they can be local
holomorphic coordinates on the cotangent bundle of
the shape sphere or of the regularized shape sphere
$\bP ({\mathcal C})$ (see the examples below).
Our computations immediately below hold for all these cases. 
 
We rescale time  and the  momenta as in equation (\ref{eq_momentumscaling}) with $Z$ replacing $Y$.
We  must also rescale angular momentum according to 
\begin{equation}\label{eq_angmomentumscaling}
  \tilde \mu = \fr{f(r)\mu}{r^2}.
\end{equation}
Then energy equations $H_{\mu}=h$ or $\tilde H_{\mu} = 0$  become
\begin{equation}\label{eq_blownupenergyreduced}
\fr12 A(X)(v^2+\tilde\mu^2)+\fr12B(X)(\a,\a) -\nu(r)V(X) = r\nu(r) C(X)
\end{equation}
where 
\begin{equation}
\label{eq_nu}
\nu = f^2/r^3 
\end{equation}
so that $\nu= 1$ for $f=r^\fr32$ and $\nu = (1+r)^{-3}$ for $f=f_2$.

In order to express the differential equations succinctly, let
 $$\tilde K = \fr12 A(X)(v^2+\tilde\mu^2)+\fr12B(X)(\a,\a)$$
 denote the blown-up kinetic energy and let
 \begin{equation}
\label{eq_phi}
 \phi(r)=-\fr12(1-r(\ln\nu)_r).
 \end{equation}
Then the equations of motion are:
  \begin{equation}
 \begin{aligned}
r' &= A(X)vr \\
v'&=  \phi(r)A(X) v^2 + 2\tilde K    - \nu(r)V\\
\tilde \mu' &= \phi(r)A(X)  v  \tilde \mu \\ 
X'&= B(X) \a \\
\a'&=\phi(r)A(X)  v\a-\tilde K_X     + \nu(r)V_X  +r\nu(r) C_X +T_{curv} \\
\end{aligned}
\end{equation}
where
$$T_{curv} = - 2 i \tilde\mu \alpha\text{ or } - 2 i \tilde\mu \tau(X)   \alpha$$
for the unregularized and regularized cases, respectively.
We remark that the $v'$ equation can also be written
$$v' = (\phi +1) A(X) v^2 + B(X)(\a,\a)+ A(X) \tilde\mu ^2  - \nu(r)V(X).$$
In these equations, we are regarding $\tilde\mu$ as a new variable subject, by definition, to the constraint 
\begin{equation}\label{eq_mutildeconstraint}
\sqrt{r}\,\tilde\mu = \sqrt{\nu(r)}\,\mu
\end{equation}
where $\mu$ is the old angular momentum constant.  This point of view is necessary to make the curvature term smooth at $r=0$.

As in section~\ref{subsec_beforereduction}, all functions of $r$ extend smoothly to $r= 0$.  If we start with one of the regularized Hamiltonians $\tilde H_\mu$ then for the resulting differential equations, all binary collisions have been regularized and the triple collision blown-up.   We obtain a flow on a manifold-with-boundary of dimension $5$ after fixing $\mu$, setting $\tilde H_{\mu} = 0$, imposing the constraint on $\tilde\mu$, the constraints that $X\in\cC$ and $\pairing{Z,X} = 0$ and passing to the quotient under complex scaling.  Binary collisions are regularized for all values of $\mu$ and if the time rescaling is done using $f_2(r)$, the flows on these manifolds will be complete.

It is well-known that  triple collisions are possible in the three-body problem only when $\mu = 0$.  In this case, equation (\ref{eq_mutildeconstraint}) shows that either
$\tilde\mu = 0$ or $r=0$.  Both of these submanifolds  are  invariant sets for the dynamical system.   The $5$-dimensional manifold-with-boundary with the above constraints and with $\tilde\mu=0$ represents the closure of zero-angular-momentum three-body problem.  The $4$-dimensional manifold where $\tilde\mu = r= 0$ forms the boundary. 
Even though orbit with $\mu\ne 0$ cannot have $\r\into 0$, the part of the collision manifold $\{r=0\}$ where $\tilde\mu \ne 0$ is relevant for studying low-angular momentum orbits passing close to triple collision \cite{MoeHeteroclinic,MoeChaotic}.

We will now present a couple of versions of the regularized, reduced and blown-up differential equations for the three-body problem.

\begin{example}[The blown-up regularized  affine  equations]
In section~\ref{sec_regularizedaffine} we used affine local coordinates on the regularized shape sphere to obtain a regularized Hamiltonian
$\tilde H(z,\z)$ with $6$-degrees of freedom.  (We wrote $z=x+i\,y, \z=\a+i\,\b$ in section~\ref{sec_regularizedaffine}.)
Comparing with the general form (\ref{eq_generalHmu}) we have
$$\begin{aligned}
A(X) &= \tau(z)\qquad B(X)(Z,Z)  = \fr{\tau}{\lambda}(1 +|z|^2)^2\,|\z|^2\\
C(X) &= h \tau(z)\qquad  V(X) = W(z).
\end{aligned}$$
For convenience, we recall from section~\ref{sec_regularizedaffine} that
\begin{equation*}
\begin{aligned}
\rho_{12}&=4|z|^2\quad \rho_{31}= |1+z^2|^2 \quad\rho_{23}= |1-z^2|^2\\
W(z) &= \fr{|X(z)|}{(1+|z|^2)^{6}}\left(m_1m_2\rho_{31}\rho_{23} + m_1m_3\rho_{12}\rho_{23} + m_2m_3\rho_{12}\rho_{31}\right) \\
|X(z)|^2 &= \fr{m_1m_2\rho_{12}^2+m_1m_3\rho_{31}^2+m_2m_3\rho_{23}^2}{m_1+m_2+m_3}\\
\tau &= \fr{\rho_{12}\rho_{31}\rho_{23}}{8(1+|z|^2)^{6}}\qquad \fr{\tau}{\lambda} =  \fr{m|X(z)|^4}{64 m_1m_2m_3\,(1+|z|^2)^{8}}.
\end{aligned}
\end{equation*}

As per the preceding subsection, we continue
to write the rescaled momentum variable as $\alpha$ (thus: $\alpha = (f/r^2) \zeta$) trusting that
there will be no confusing with the previous use of $\alpha$.
The rescaled kinetic energy satisfies
$$2\tilde K = \tau v^2 +  \tau \tilde \mu^2 + \fr{\tau}{\lambda}(1+|z|^2)^2 \, |\a|^2.$$

Then the regularized, blown-up equations read:
 \begin{equation}
 \begin{aligned}
r' &= \tau(z) vr \\
v'&=  \phi(r) \tau(z) v^2 + 2\tilde K    - \nu(r)W(z)\\
\tilde \mu' &= \phi (r) \tau (z) v  \tilde \mu \\
z'&= (1 + |z|^2)^2  \a \\
\a'&=\phi(r) \tau(z)   v\a- \tilde K_z + \nu (r) W_z  + r\nu(r) h \tau_z (z) - 2 i \tilde \mu \tau(z) \alpha \\
\end{aligned}
\end{equation}
The possibilities for $\nu(r), \phi(r)$ are described
in the previous subsection, in equations (\ref{eq_nu}, \ref{eq_phi}). 

We have $7$ variables, $(r,v,\tilde\mu,z,\a)\in [0,\infty)\times\R\times\R\times\C\times\C$.  The constraints are
$$\fr12 \tau(z)(v^2+\tilde\mu^2)+\fr12 \fr{\tau}{\lambda}(1 +|z|^2)^2\,|\a|^2  - \nu(r)W(z) = r\nu(r)\tau(z)h$$
and
$$\sqrt{r}\,\tilde\mu = \sqrt{\nu(r)}\,\mu.$$
\end{example}

\begin{example}[The blown-up regularized spherical equations]
In section~\ref{sec_regularizedshapesphereround} we used spherical-homogeneous variables $c=(c_1,c_2,c_3)$ to give a global representation of the regularized shape sphere.   We found a regularized Hamiltonian $\tilde H_\mu(r,c,p_r,\g)$. 
Comparing with the general form (\ref{eq_generalHmu}) we have
$$\begin{aligned}
A(X) &= \tau(c)\qquad B(X)(Z,Z)  =  2\fr{\tau}{\lambda}|c|^2 |\gamma|^2\\
C(X) &= h \tau(c)\qquad  V(X) = W(c).
\end{aligned}$$
We recall the formulas
\begin{equation*}
\begin{aligned}
\rho_{12}&= c_{31}^2+c_{23}^2\quad\rho_{31}= c_{12}^2+c_{23}^2\quad\rho_{23}= c_{12}^2+c_{31}^2\\
W(c) &= \fr{|X(c)|}{8|c|^6}\left(m_1m_2\rho_{31}\rho_{23} + m_1m_3\rho_{12}\rho_{23} + m_2m_3\rho_{12}\rho_{31}\right) \\
|X(c)|^2 &= \fr{m_1m_2\rho_{12}^{2}+m_1m_3\rho_{31}^{2}+m_2m_3\rho_{23}^{2}}{m_1+m_2+m_3}\\
\tau &= \fr{\rho_{12}\rho_{31}\rho_{23}}{8|c|^6}\qquad \fr{\tau}{\lambda} =  \fr{m|X(c)|^4}{64 m_1m_2m_3\,|c|^8}.
\end{aligned}
\end{equation*}
With $\alpha = (f/r^2) \gamma$, the rescaled kinetic energy satisfies
$$2\tilde K = \tau v^2 +  \tau \tilde \mu^2 + 2\fr{\tau}{\lambda}|c|^2 \, |\a|^2.$$

Then the regularized, blown-up equations read:
 \begin{equation}\label{eq_blownupspherical}
 \begin{aligned}
r' &= \tau(c) vr \\
v'&=  \phi(r) \tau(c) v^2 + 2\tilde K    - \nu(r)W(c)\\
\tilde \mu' &= \phi(r) \tau v  \tilde \mu \\
c'&=  2 \fr{\tau (c)}{\lambda (c)}|c|^2  \a \\
\a'&=\phi(r) \tau(c)   v\a- \tilde K_c + \nu W_c  + r\nu(r) h \tau_c (c)+ \fr{2  \tilde \mu \tau(c)}{|c|}\alpha \times c\\
\end{aligned}
\end{equation}

 We have $9$ variables, $(r,v,\tilde\mu,c,\a)\in [0,\infty)\times\R\times\R\times\R^3_0\times\R^3$.  However, $(c,\a)$ are homogeneous variables.  They satisfy $\pairing{\a,c} = 0$ and the equations are invariant under the scaling $(c, \alpha) \to (k c, \fr{1}{k} \a)$.   Taking this into account, we  have an induced system on
the $7$-dimensional quotient space $[0,\infty)\times\R\times\R\times T^*\S^2$.   
The energy and angular momentum constraints are
\begin{equation}\label{eq_blownupsphericalenergy}
\fr12 \tau(c)(v^2+\tilde\mu^2)+ \fr{\tau}{\lambda}|c|^2\,|\a|^2  - \nu(r)W(c) = r\nu(r)\tau(c)h
\end{equation}
and
$$\sqrt{r}\,\tilde\mu = \sqrt{\nu(r)}\,\mu$$
giving a subvariety of dimension $5$.

A nice alternative to the quotient construction is just to observe that  $\pairing{\a,c}=0$ implies that $|c|$ is invariant under the differential equations (\ref{eq_blownupspherical}).   Instead of quotienting by the scaling symmetry we can simply restrict $c$ to the unit sphere.   Let
$${\mathcal{M}}(h,\mu) = \{(r,v,\tilde\mu,c,\a): |c| = 1,\pairing{\a,c}=0, \tilde\mu = \sqrt{\fr{\nu(r)}{r}}\,\mu, (\ref{eq_blownupsphericalenergy})\text{ holds }\}.$$
Then ${\mathcal{M}}(h,\mu) $ is a $5$-dimensional submanifold (or subvariety when $\mu = 0$) of  $[0,\infty)\times\R\times\R\times\R^3_0\times\R^3$ which is invariant under (\ref{eq_blownupspherical}).  The flow on ${\mathcal{M}}(h,\mu) $ globally represents the planar three-body problem reduced by translations and rotations, with all binary collisions regularized and with triple collision blown-up.

\end{example}

\bibliographystyle{amsplain}

\end{document}